\documentclass[12pt]{article}

\usepackage{amsmath, amssymb, amsthm, xypic, stmaryrd,graphicx}
\usepackage[all]{xy}

\DeclareMathOperator{\supp}{supp}

\DeclareMathOperator{\alg}{alg}
\DeclareMathOperator{\red}{red}

\DeclareMathOperator{\Av}{Av}

\DeclareMathOperator{\ind}{index}
\DeclareMathOperator{\Kind}{\textit{K}-index}

\DeclareMathOperator{\Realpart}{Re}

\DeclareMathOperator{\GL}{GL}

\DeclareMathOperator{\Spin}{Spin}

\DeclareMathOperator{\Sym}{Sym}

\DeclareMathOperator{\Ad}{Ad}

\DeclareMathOperator{\tr}{tr}

\DeclareMathOperator{\grad}{grad}

\DeclareMathOperator{\Crit}{Crit}

\DeclareMathOperator{\restr}{restr}
\DeclareMathOperator{\image}{image}

\newcommand{\beq}{\begin{equation}}
\newcommand{\eeq}{\end{equation}} 
\newcommand{\bea}{\begin{eqnarray}}
\newcommand{\eea}{\end{eqnarray}}

\begin{document}

\theoremstyle{plain}
\newtheorem{theorem}{Theorem}[section]
\newtheorem{thm}{Theorem}[section]
\newtheorem{lemma}[theorem]{Lemma}
\newtheorem{proposition}[theorem]{Proposition}
\newtheorem{prop}[theorem]{Proposition}
\newtheorem{corollary}[theorem]{Corollary}
\newtheorem{conjecture}[theorem]{Conjecture}

\theoremstyle{definition}
\newtheorem{definition}[theorem]{Definition}
\newtheorem{defn}[theorem]{Definition}
\newtheorem{example}[theorem]{Example}
\newtheorem{remark}[theorem]{Remark}
\newtheorem{rem}[theorem]{Remark}

\newcommand{\C}{\mathbb{C}}
\newcommand{\R}{\mathbb{R}}
\newcommand{\Z}{\mathbb{Z}}
\newcommand{\N}{\mathbb{N}}

\newcommand{\field}[1]{\mathbb{#1}}
\newcommand{\bZ}{\field{Z}}
\newcommand{\bR}{\field{R}}
\newcommand{\bC}{\field{C}}
\newcommand{\bN}{\field{N}}
\newcommand{\bT}{\field{T}}

\newcommand{\cB}{{\mathcal{B} }}
\newcommand{\cK}{{\mathcal{K} }}
\newcommand{\cF}{{\mathcal{F} }}
\newcommand{\cO}{{\mathcal{O} }}
\newcommand{\cE}{\mathcal{E}}
\newcommand{\cS}{\mathcal{S}}

\newcommand{\HH}{{\mathcal{H} }}
\newcommand{\tilH}{\widetilde{\HH}}
\newcommand{\HX}{\HH_X}
\newcommand{\Hpi}{\HH_{\pi}}
\newcommand{\HHpi}{\HH \otimes \HH_{\pi}}
\newcommand{\Ltwopi}{L^2_{\pi}(X, \HHpi)}

\newcommand{\KK}{K\!K}

\newcommand{\D}{D \hspace{-0.27cm }\slash}
\newcommand{\Dsmall}{D \hspace{-0.19cm }\slash}

\newcommand{\mybigwedge}{\textstyle{\bigwedge}}

\newcommand{\CGmax}{C^*_{G, \max}}
\newcommand{\CGred}{C^*_{G, \red}}
\newcommand{\CGalg}{C^*_{G, \alg}}
\newcommand{\CGker}{C^*_{G, \ker}}
\newcommand{\Cmax}{C^*_{\max}}
\newcommand{\Cred}{C^*_{\red}}
\newcommand{\Calg}{C^*_{\alg}}
\newcommand{\Cker}{C^*_{\ker}}
\newcommand{\tilCalg}{\widetilde{C}^*_{\alg}}
\newcommand{\Cpiker}{C^*_{\pi, \ker}}
\newcommand{\Cpialg}{C^*_{\pi, \alg}}
\newcommand{\Cpimax}{C^*_{\pi, \max}}

\newcommand{\Avpi}{\Av^{\pi}}

\newcommand{\Gtc}{\Gamma^{\infty}_{tc}}
\newcommand{\tilD}{\widetilde{D}}
\newcommand{\XoneH}{X^{\HH}_1}

\newcommand{\XX}{\mathfrak{X}}

\def\kt{\mathfrak{t}}
\def\kk{\mathfrak{k}}
\def\kp{\mathfrak{p}}
\def\kg{\mathfrak{g}}
\def\kh{\mathfrak{h}}

\newcommand{\gse}{\mathfrak{g}^*_{\mathrm{se}}}

\newcommand{\pilamrho}{[\pi_{\lambda+\rho}]}

\newcommand{\Trestr}{\mathcal{T}_{\restr}}
\newcommand{\TdN}{\mathcal{T}_{d_N}}

\newcommand{\omG}{\om/\hspace{-1mm}/G}
\newcommand{\om}{\omega} \newcommand{\Om}{\Omega}

\newcommand{\QcwR}{quantization commutes with reduction}

\def\kt{\mathfrak{t}}
\def\kk{\mathfrak{k}}
\def\kp{\mathfrak{p}}
\def\kg{\mathfrak{g}}
\def\kh{\mathfrak{h}}

\newcommand{\ddt}{\left. \frac{d}{dt}\right|_{t=0}}

\newcommand{\Sj}{\sum_{j=1}^{d_G}}
\newcommand{\Sk}{\sum_{k=1}^{d_M}}
\newcommand{\Sjk}{\Sj \Sk }


\title{Geometric quantization and families of inner products}
\author{Peter Hochs\footnote{University of Adelaide, \texttt{peter.hochs@adelaide.edu.au}},
Varghese Mathai\footnote{University of Adelaide, \texttt{mathai.varghese@adelaide.edu.au}}}

%
%



 \maketitle

\begin{abstract}
We formulate a {\em quantization commutes with reduction} principle in the setting where the Lie group $G$, the symplectic manifold it acts on, and the orbit space of the action may all be noncompact. It is assumed that  the action is proper, and the zero set of a deformation vector field, associated to the momentum map and an equivariant family of inner products on the Lie algebra $\mathfrak{g}$ of $G$, is $G$-cocompact. The central result establishes an asymptotic version of this quantization commutes with reduction  principle.
Using an equivariant family of inner products on $\mathfrak{g}$ instead of a single one makes it possible to handle both noncompact groups and manifolds, by extending Tian and Zhang's Witten deformation approach to the noncompact case.
 \end{abstract}


\tableofcontents

\section{Introduction}

\subsection{Background}

Geometric quantization and the \emph{quantization commutes with reduction} principle have been studied intensively for decades. Geometric quantization has its origins in physics, where it is a method to construct the quantum mechanical description of a physical system from its classical mechanical description. The quantization commutes with reduction principle states that geometric quantization is compatible with the ways symmetry works in classical and quantum mechanics.

The mathematical language of classical mechanics is symplectic geometry (or more generally, Poisson geometry). Symmetry in quantum mechanics leads to a unitary representation of the symmetry group involved. The quantization commutes with reduction principle has revealed deep connections between symplectic geometry and the theory of unitary representations, with various applications to representation theory and physics.

In their 1982 paper \cite{GuSt82}, Guillemin and Sternberg conjectured that quantization commutes with reduction, and proved this for compact Lie groups acting on compact K\"ahler manifolds, under a positivity assumption. A definition of geometric quantization that is valid for compact Lie groups acting on compact, possibly non-K\"ahler symplectic manifolds is attributed to Bott. 
Let $(M, \omega)$ be a compact symplectic manifold, on which a compact Lie group $K$ acts, preserving $\omega$. Let $L \to M$ be a $K$-equivariant Hermitian line bundle whose first Chern class is $[\omega]$. Let $D^L$ be the Dolbeault- or $\Spin^c$-Dirac operator on $M$, coupled to $L$. If all structures involved in the definition of this operator are $K$-invariant, then $D^L$ is $K$-equivariant. Bott's definition of the geometric quantization $Q_K(M, \omega)$ is
\begin{equation} \label{eq quant bott}
Q_K(M, \omega) := \Kind (D^L) \quad \in R(K).
\end{equation}
Here $R(K)$ is the representation ring of $K$, and
\begin{equation} 
\Kind (D^L) := [\ker D^L_+] - [\ker D^L_-],
\end{equation}
with $D^L_{\pm}$ the even and odd parts of $D^L$. Since $M$ and $K$ are compact, this $K$-index is indeed a well-defined element of $R(K)$. 

If the action of $K$ on $M$ is \emph{Hamiltonian}, there is a \emph{momentum map}
\begin{equation} 
\mu: M \to \kk^*,
\end{equation}
where $\kk^*$ is the dual of the Lie algebra $\kk$ of $K$. If $0 \in \kk^*$ is a regular value of $\mu$, then the space
\begin{equation} 
M_0 := \mu^{-1}(0)/K
\end{equation}
is an orbifold, since it can be shown that $K$ acts on $\mu^{-1}(0)$ with finite stabilizers. For simplicity, one can assume that these stabilizers are trivial, so that $M_0$ is a smooth manifold. The symplectic form $\omega$ naturally induces a symplectic form $\omega_0$ on $M_0$. The symplectic manifold $(M_0, \omega_0)$ is called the \emph{symplectic reduction} or \emph{Marsden--Weinstein reduction} \cite{MW} of $(M, \omega)$. The geometric quantization of $(M_0, \omega_0)$ is defined as the index of the Dirac operator $D^{L_0}$ on $M_0$, coupled to the line bundle $L_0 \to M_0$  induced by $L$.

In terms of Bott's definition of geometric quantization, Guillemin and Sternberg's conjecture that quantization commutes with reduction states that the following diagram commutes:
\begin{equation} \label{diag quant red cpt}
\xymatrix{ 
K \circlearrowright (M, \omega) \ar@{|->}[rr]^-{Q_K} \ar@{|->}[d]_-{\text{reduction}}^-{R} 	
	& & \Kind (D^L)  \ar@{|->}[d]^-{\text{reduction}}_-{R}  \\ 
(M_{0}, \omega_{0}) \ar@{|->}[r]^-{Q} & Q(M_{0}, \omega_{0}) \ar@{=}[r]& \bigl(\Kind (D^L) \bigr)^K.
}
\end{equation}
That is, the quantization of $(M_0, \omega_0)$ is the $K$-invariant part of the quantization of $(M, \omega)$. 
The Guillemin--Sternberg conjecture in this generality was first proved 16 years after Guillemin and Sternberg's paper, 
by Meinrenken \cite{M} (see also \cite{M96}), and by Meinrenken and Sjamaar \cite{MS} for singular values of the momentum map. Other proofs were given by Tian and Zhang \cite{TZ97,TZ98} and by Paradan \cite{Paradan1}.

\subsection{Noncompact groups and manifolds}

After these results, the natural desire arose to generalize the quantization commutes with reduction principle to noncompact manifolds and groups. Such a generalization is very relevant to
\begin{itemize}
\item \emph{physics}, since most classical mechanical phase spaces (such as cotangent bundles) are not compact; and to
\item \emph{representation theory}, since the representation theory for noncompact groups is much more intricate than for compact groups, and could benefit greatly from such a principle in the noncompact case.
\end{itemize}
However, even stating a quantization commutes with reduction principle in the noncompact setting proved highly challenging:
\begin{itemize}
\item if the manifold $M$ is noncompact, then the kernel of the Dirac operator $D^L$ is infinite-dimensional in general, so its index is not well-defined;
\item if the compact group $K$ is replaced by a noncompact group $G$, the finite-dimensional representations of $G$, which make up $R(G)$, are not the interesting ones.
\end{itemize}

Ma and Zhang \cite{Ma-Zhang, Ma-Zhang2} solved an extended version of Vergne's conjecture \cite{VergneICM} on quantization commutes with reduction for
 compact groups $G=K$ acting on noncompact manifolds, and later on Paradan \cite{Paradan2, Paradan3} gave a new proof of it.  They define quantization as an element of the \emph{generalized representation ring} $R^{-\infty}(K)$, by taking localized indices of the Dirac operator $D^L$ on expanding families of suitable relatively compact open subsets of $M$. Their approach applies to several interesting examples, such as 
 cotangent bundles of a homogeneous spaces of compact Lie groups  \cite{Paradan2} and  coadjoint orbits associated to holomorphic discrete series representations of  reductive Lie groups \cite{Paradan3}. 

Landsman \cite{Landsman} stated a conjecture\footnote{called the `Hochs--Landsman conjecture' in \cite{MZ}} for noncompact groups $G$ and manifolds $M$, assuming that the action is \emph{cocompact}, i.e.\ the orbit space $M/G$ is compact. He used the \emph{analytic assembly map} from the Baum--Connes conjecture 
to define geometric quantization. This assembly map generalizes the $K$-index, and is a map
\begin{equation} \label{eq ass map}
\mu_M^G: K_*^G(M) \to K_*(C^*(G)).
\end{equation}
Here $K_*^G(M)$ is the equivariant $K$-homology  of $M$, 
 and $K_*(C^*(G))$ is the $K$-theory of the $C^*$-algebra $C^*(G)$ of $G$. Landsman defined geometric quantization as
\begin{equation} \label{eq quant klaas}
Q_G(M, \omega) := \mu_M^G\bigl[D^L\bigr] \quad \in K_*(C^*(G)),
\end{equation}
where $[D^L\bigr]  \in K_*^G(M)$ is the $K$-homology class naturally defined by the Dirac operator $D^L$. 
There is a quantum reduction map $R_G: K_0(C^*(G)) \to \Z$, induced by averaging over the group $G$.  Landsman's 
quantization commutes with reduction conjecture asserts that 
\begin{equation} \label{eq Landsman}
R_G(Q_G(M, \omega)) = Q(M_0, \omega_0) \quad \in \Z.
\end{equation}
This conjecture was proved in a special case by Hochs and Landsman in \cite{HL}, and extensions to reduction at nontrivial representations were proved in \cite{Hochs,HochsPS}. Mathai and Zhang  \cite{MZ} proved an asymptotic version of the conjecture, for large tensor powers of the prequantum line bundle $L$. 

Despite these successes, the general noncompact case remained intractable by the methods developed so far. Indeed, from the physics point of view, the simplest classical mechanical system is a free particle moving in Euclidean space $\R^n$. This system is described by the action by $\R^n$ on its cotangent bundle $T^*\R^n$. Since both the group acting and the orbit space are noncompact, neither the techniques of Ma and Zhang, nor those of Landsman can be used. This action is one particular example where the results of this paper apply, as discussed in Subsection \ref{sec special}.


\subsection{Summary of results and strategy of proof}

The aim of the present paper is to generalize the quantization commutes with reduction principle to a general noncompact setting. 
Our approach is to extend   the Tian--Zhang \cite{TZ98} Witten deformation method \cite{Witten} to the noncompact case. This method involves a deformation  $D^{L^p}_t$ of the Dirac operator $D^{L^p}$, which depends on a real deformation parameter $t$. Here, for any $p\ge 1$, the line bundle $L^p \to M$ is the $p$'th tensor power of the line bundle $L$, which is a prequantum line bundle for the symplectic manifold $(M, p\omega)$.
One difference with the method in \cite{TZ98} is that the group $G$ is noncompact,
and generally does not have an inner product on the dual of its Lie algebra $\kg^*$ that is invariant under the coadjoint action of $G$.  

To overcome this, our first innovation is to work with \emph{families} of inner products on $\kg^*$, parametrized by $M$, which have a natural $G$-invariance property. Such families were used by Kasparov  in Section 6 of \cite{Kasparov2012} for a different purpose. Using a family of inner products allows one to define a $G$-invariant  Hamiltonian function $\HH$ as the norm squared of the momentum map $\mu$ and a $G$-invariant vector field 
$\XoneH$, which is used to deform the Dirac operator. The zero set of this vector field is assumed to be cocompact, which 
implies that the symplectic reduction $M_0$ at zero of the action is compact. 

The $G$-invariant part of the quantization of the action by $G$ on $(M, \omega)$ may then be defined as the integer
\[
Q(M, \omega)^G := \dim \left(\ker_{L^2_T} \bigl( (D^{L}_t)_+ \bigr)  \right)^G - \dim \left(\ker_{L^2_T} \bigl( (D^{L}_t)_- \bigr)  \right)^G,
\]
for $t$ large enough. Here $\ker_{L^2_T}$ denotes the space of sections in the kernel of an operator that are square-integrable transversally to orbits, in an appropriate sense. The first main result of this paper is that invariant quantization is well-defined in this way. 
\begin{theorem}\label{thm quant well defd intro}
For $t$ large enough, the $G$-invariant part $\bigl(\ker_{L^2_T}(D^{L}_t) \bigr)^G$ of the vector space $\ker_{L^2_T}(D^{L}_t)$ is finite-dimensional. 
\end{theorem}
The proof of this fact involves some index theory on Sobolev spaces created from $G$-invariant sections, and a generalization of the Anghel--Gromov--Lawson criterion \cite{Anghel,GL} for Fredholmness. For compact groups $G$, Braverman used a similar approach in \cite{Braverman}. Using techniques from \cite{Braverman} and the present paper, Braverman also developed an approach for noncompact groups \cite{Braverman2}.

 In terms of this definition of invariant quantization, one can state the quantization commutes with reduction principle as follows. (See Subsections \ref{sec Dirac} and \ref{sec result} for details.)
\begin{conjecture}[Quantization commutes with reduction] \label{con [Q,R]=0} Under the assumptions that the Hamiltonian $G$-action on the prequantizable symplectic manifold $(M, \omega)$ is proper, $0$ is a regular value of the momentum map, and the 
 zero set of the $G$-invariant vector field $\XoneH$ is $G$-cocompact, one has the equality
\begin{equation}  \label{eq [Q,R]=0 conjecture}
Q(M, \omega)^G = Q(M_0, \omega_0),
\end{equation}
where $Q(M_0, \omega_0)$ is the quantization of the symplectic reduction at zero.
\end{conjecture}
If $M/G$ is compact, this conjecture reduces to Landsman's conjecture \eqref{eq Landsman} (see Corollary \ref{cor MZ}), which in turn reduces to commutativity of \eqref{diag quant red cpt} if both $G$ and $M$ are compact.

The second main result in this paper establishes an asymptotic version of this principle, under the simplifying assumption that $G$ acts freely on $\mu^{-1}(0)$ (rather than just locally freely).
\begin{theorem} \label{thm [Q,R]=0 intro}
If $G$ acts freely on $\mu^{-1}(0)$, the equality \eqref{eq [Q,R]=0 conjecture} holds for large enough multiples of $\omega$:
\begin{equation} \label{eq [Q,R]=0 intro}
Q(M, p\omega)^G = Q(M_0, p\omega_0),
\end{equation}
for any integer $p$ at least equal to a minimal value $p_0$.
\end{theorem}
In the special case when $M/G$ is compact, we recover the main result in \cite{MZ}.

As in \cite{TZ98}, the proofs of the main results start with a Bochner-type formula for the square of the deformed Dirac operator $D^{L^p}_t$ on $G$-invariant sections:
\begin{equation}
\left( D^{L^p}_t\right)^2 = \left( D^{L^p}\right)^2 + tA + 4\pi p t\HH + \frac{t^2}{4} \|\XoneH\|^2,
\end{equation}
where the tensor term $A$ is a generalization of the one in \cite{TZ98}. 
It can be decomposed as $A=A_1+A_2 + A_3$ where $A_1$ is the Tian--Zhang tensor and $A_2, A_3$ are new tensors that vanish in the setting 
of \cite{TZ98}.
In the non-cocompact case, it is possible
for $A$ to be unbounded, so it appears at first that the method in \cite{TZ98} does not work. However, we use
the flexibility that one has in choosing an equivariant family of inner products on $\kg^*$, and make
a judicious choice of such a family to bound $A$, which is a key ingredient of our approach.

 

Theorem \ref{thm quant well defd intro} which proves the Fredholm property of the deformed Dirac operator $D^{L^p}_t$ on $G$-invariant sections, 
also shows that for $t$ and $p$ large, the elements of the vector space $\bigl(\ker_{L^2_T}(D^{L}_t) \bigr)^G$ localize to a relatively
$G$-cocompact open subset of $M$. 
As in \cite{TZ98} and \cite{MZ}, adaptations of  a result by Bismut and Lebeau \cite{BL91} are then used to
deduce the equality in equation \eqref{eq [Q,R]=0 intro}.
As a comparison, in \cite{MZ}, 
Mathai and Zhang  used any inner product on $\kg^*$, so that the Hamiltonian function given by the norm squared of the momentum map
was no longer $G$-invariant. They used a weighted average of the Hamiltonian vector field, which is $G$-invariant,
but which is not the Hamiltonian vector field of a $G$-invariant function. 
 In the present paper, we use instead a $G$-invariant family of inner products on $\kg^*$ 
 and so the Hamiltonian function given by the norm squared of the momentum map
is $G$-invariant.
 
In Subsection \ref{sec special}, it is shown that the techniques developed here apply for example to physical systems where the configuration space is a Lie group $G$, acted on by the group itself via left multiplication. Then $M = T^*G$ is the cotangent bundle of $G$, and the orbit space of the action is noncompact. The zero set of the vector field $\XoneH$ is cocompact however, so Theorem \ref{thm [Q,R]=0} applies. This in particular applies to the case of a free particle in $\R^n$ mentioned above, where $G = \R^n$. Other examples discussed in Subsection \ref{sec special} include the case
when $M/G$ is compact, and also the case when $G$ itself is compact, which is relevant to the Vergne 
conjecture that was completely solved in \cite{Ma-Zhang, Ma-Zhang2} using a very different index theorem. 
Finally, the cocompactness assumption in \cite{HL, Landsman, MZ} precluded any form of the shifting trick for noncompact groups. In the present setting, a version of the shifting trick holds.
 

\subsection{Outline of this paper}

The key ingredient of our method is the use of $G$-invariant families of inner products on $\kg^*$. These are introduced in Section \ref{sec X1H}. The vector fields defined via these families are also discussed.

The main results of this paper are Theorems \ref{thm quant well defd intro} and \ref{thm [Q,R]=0 intro}, which are formulated 
in a precise way in Subsection \ref{sec result}. The index theory used to show that invariant quantization is well-defined, is developed in Section \ref{sec index}.

In Section \ref{sec localise}, we will see that the kernels of the deformed Dirac operators we use localize in a suitable way. The argument for this localization is based on an explicit computation of the square of the deformed Dirac operator in Section \ref{sec Bochner}.  A relation between these deformed Dirac operators and certain Dirac operators on the symplectic reduction, proved in Section \ref{sec Dirac M M0}, then allows us to complete the proof that quantization commutes with reduction.

A version of elliptic regularity is proved in  Appendix \ref{app reg L2t}.  Appendix \ref{app Bochner} and Appendix \ref{app bound A} contain some computations and estimates used in the main text, involving deformed Dirac operators.

\subsection{Acknowledgements}

The authors would like to thank Maxim Braverman for carefully reading a preliminary version of this paper, and giving useful advice. The authors also benefitted from discussions with Gert Heckman, Nigel Higson, Klaas Landsman, Paul--\'Emile Paradan, Maarten Solleveld and Weiping Zhang. 

The first author was supported by the Alexander von Humboldt foundation and by the European Union, through  Marie Curie fellowship PIOF-GA-2011-299300. A collaboration visit of the first author to the second was funded by the Australian Research Council, through Discovery Project DP110100072. The second author
acknowledges funding by the Australian Research Council, through Discovery Project DP130103924.

\subsection{Notation and conventions}

For a smooth manifold $M$, we denote the spaces of smooth functions, smooth $k$-forms and smooth vector fields on $M$ by $C^{\infty}(M)$, $\Omega^k(M)$ and $\XX(M)$, respectively. The value of a vector field $v$ on $M$ at a point $m$ will be denoted by $v_m$ or $v(m)$, whichever seems clearer. 

If $E\to M$ is a vector bundle, the space of smooth sections  of $E$ is denoted by $\Gamma^{\infty}(M, E)$ or $\Gamma^{\infty}(E)$. We write $\Omega^k(M; E)$ for the space of smooth sections of $\mybigwedge^kT^*M \otimes E$. For almost complex manifolds, $\Omega^{0, k}(M)$ and $\Omega^{0, k}(M; E)$ denote the analogous spaces of $(0,k)$-forms.

A subscript $c$ denotes compactly supported functions or sections. In the equivariant case, where a group $G$ acts on the relevant structures, a superscript $G$ denotes the space of $G$-invariant elements. 

The Lie algebra of a Lie group $G$ will be denoted by $\kg$. We will denote the dimension of a manifold $M$ by $d_M$. 

Suppose $G$ acts continuously on a topological space $X$, and let $g \in G$ and $x \in X$. Then $G_x$ and $\kg_x$ are the stabiliser group and algebra of $x$, respectively. The homeomorphism from $X$ to itself defined by $g$ will also be denoted by $g$. In particular, if $X$ is a smooth manifold and the action is smooth, we have the tangent map
\[
T_xg: T_xX \to T_{gx}X
\] 
of the diffeomorphism $g: X \to X$.
The action will be called \emph{cocompact} if the orbit space $X/G$ is compact. A subset $Y \subset X$ is called \emph{relatively cocompact} if its image under the quotient map is relatively compact.

We will write $\N = \{1, 2, 3, \ldots\}$ for the set of natural numbers without the number $0$.

\section{Families of metrics on $\kg^*$} \label{sec X1H}

The methods used in this paper are based on deforming Dirac operators using a certain $G$-invariant vector field $\XoneH$. This vector field is similar to the Hamiltonian vector field of the norm-squared function of a given momentum map, used in \cite{TZ98}. 
If a Lie group $G$ is noncompact, an $\Ad^*(G)$-invariant inner product on the dual $\kg^*$ of the Lie algebra $\kg$
may not exist. Then the Hamiltonian vector field of the norm-squared function of a momentum map is not $G$-invariant in general. To solve this problem, we work with \emph{families} of inner products on $\kg^*$, parametrized by a manifold on which $G$ acts.
Such a family of inner products allows one to define the vector field $\XoneH$ mentioned above. The zero set of this vector field will later be assumed to be cocompact. 

\subsection{Norms of momentum maps} \label{sec vector fields}

Let $(M, \omega)$ be a symplectic manifold, and let $G$ be a Lie group. Let a proper Hamiltonian action by $G$ on $M$ be given, and let
\[
\mu: M \to \kg^*
\]
be a momentum map. We will use the sign convention that for all $X \in \kg$,
\begin{equation} \label{eq def mom}
d\mu_X = \omega(X^M, \relbar),
\end{equation}
where $\mu_X$ denotes the pairing of $\mu$ and $X$, and $X^M$ is the vector field on $M$ induced by $X$.

As in \cite{TZ98}, we will consider the norm-squared function of $\mu$ and use this to deform a Dirac operator on $M$. We would like this function to be $G$-invariant, so that the resulting deformation is again a $G$-equivariant operator. Since there is no $\Ad^*(G)$-invariant inner product on the dual $\kg^*$ of the Lie algebra $\kg$ in general, we consider a smooth family of inner products $\{(\relbar, \relbar)_m \bigr \}_{m \in M}$ on $\kg^*$, which is \emph{$G$-equivariant}, in the sense that for all $m \in M$, $g \in G$ and $\xi, \xi' \in \kg^*$,
\begin{equation} \label{eq equivar metric}
\bigl(\Ad^*(g)\xi, \Ad^*(g)\xi'\bigr)_{g\cdot m} = (\xi, \xi')_m.
\end{equation}
Put differently, the inner products $(\relbar, \relbar)_m$ define a $G$-invariant smooth metric on the $G$-vector bundle $M \times \kg^* \to M$, equipped with the $G$-action
\begin{equation} \label{eq action M g star}
g\cdot (m, \xi) = (g\cdot m, \Ad^*(g)\xi),
\end{equation}
for $g \in G$, $m \in M$ and $\xi \in \kg^*$. Such a metric was used by Kasparov in Section 6 of \cite{Kasparov2012} in a different context, and always exists.
\begin{lemma}\label{lem metric}
There is a metric on the trivial bundle $M \times \kg^* \to M$, which is invariant with respect to the $G$-action \eqref{eq action M g star}.
\end{lemma}
\begin{proof}
By Palais's theorem \cite{Palais}, the proper $G$-manifold $M\times \kg^*$ has a $G$-invariant Riemannian metric. The vector bundle $M \times \kg^*$ embeds into the restriction of $T(M\times \kg^*)$ to $M\times \{0\}$, via
\[
(m, \xi) \mapsto (0, \xi) \quad \in T_mM \times \kg^* = T_{(m,0)}(M\times \kg^*),
\]
for $m \in M$ and $\xi \in \kg^*$. 
Restricting the Riemannian metric on $M\times \kg^*$ to the subbundle $M\times \kg^*$ of $T(M\times \kg^*)|_{M\times \{0\}}$ in this way, one obtains the desired metric.
\end{proof}

Form now on, let $\{(\relbar, \relbar)_m \bigr \}_{m \in M}$ be a $G$-invariant metric on $M\times \kg^* \to M$, and let $\{ \|\cdot \|_m \}_{m \in M}$ be the associated family of norms on $\kg^*$. Consider the function $\HH \in C^{\infty}(M)$ defined by
\begin{equation} \label{eq def H}
\HH(m) = \|\mu(m)\|_m^2.
\end{equation}
It follows from equivariance of $\mu$ and the property \eqref{eq equivar metric} of the family of inner products on $\kg^*$, that $\HH$ is a $G$-invariant function on $M$. 

Consider the auxiliary function $\widetilde{\HH} \in C^{\infty}(M \times M)$ defined by
\[
\widetilde{\HH} (m, m') = \|\mu(m)\|_{m'}^{2}.
\]
We write $d_1 \HH $ and $d_2 \HH $ for the derivatives of $\widetilde{\HH}$ with respect to the first and second coordinates:
\[
\begin{split}
(d_1\HH)_m  &:= d_m\bigl(m' \mapsto \widetilde{\HH}(m', m) \bigr); \\
(d_2 \HH)_m  &:= d_m\bigl(m' \mapsto \widetilde{\HH}(m, m') \bigr),
\end{split}
\]
for any $m \in M$.
In terms of these one-forms on $M$, one has
\begin{equation}\label{eq dH}
d\HH = d_1\HH+d_2 \HH.
\end{equation}

\subsection{Two vector fields} \label{sec X12H}

Important roles will be played by the vector fields $X^{\HH}_j$ on $M$ determined by
\begin{equation} \label{eq def XjH}
d_j\HH =  \omega(X^{\HH}_j, \relbar) \quad \in \Omega^1(M).
\end{equation}
The Hamiltonian vector field $X^{\HH}$ of $\HH$ decomposes as
\begin{equation} \label{eq decomp XH}
X^{\HH} = X^{\HH}_1 + X^{\HH}_2.
\end{equation}
Note that $X^{\HH}_1$ and $X^{\HH}_2$ are not quite Hamiltonian vector fields. But they turn out to have similar useful properties.

One of these is $G$-invariance. This property will mean that one does not need to average them as in \cite{MZ}. By $G$-invariance of $\omega$, it is equivalent to $G$-invariance of $d_1\HH$ and $d_2\HH$. This follows from the following fact.
\begin{lemma}
Let $M$ be a manifold on which a group $G$ acts. Let $F: M \times M \to \R$ be a smooth function which is invariant under the diagonal action by $G$ on $M \times M$. Then the one-forms\footnote{The notation $d_jF$ is not quite consistent with the notation $d_j\HH$, since $\HH$ is a function on $M$. But the notation is hopefully self-explanatory.} $d_1F$ and $d_2F$ on $M$ are $G$-invariant.
\end{lemma}
\begin{proof}
We will prove the claim for $d_1F$. Let $m \in M$, $g \in G$, and let $\gamma$ be a curve in $M$ with $\gamma(0) = m$. Then
\[
\begin{split}
\bigl\langle (g^*(d_1F))_m, \gamma'(0) \bigr\rangle &= \langle (d_1F)_{gm}, T_mg(\gamma'(0)) \rangle \\
	&= \ddt F(g\gamma(t), gm)\\
	&= \ddt F(\gamma(t), m) \\
	&= \langle (d_1F)_m, \gamma'(0) \rangle.
\end{split}
\]
\end{proof}

It will be useful to have explicit expressions for the vector field $\XoneH$. For any map $f: M \to \kg^*$, we will write $f^*$ for the map from $M$ to $\kg$ determined by
\begin{equation} \label{eq dual f}
(f(m), \xi)_{m} = \langle \xi, f^*(m)\rangle,
\end{equation}
for all $\xi \in \kg^*$. This dual map induces a vector field $V_f$ on $M$, by
\begin{equation} \label{eq Vf}
V_f(m) := f^*(m)^M_m = \ddt \exp(tf^*(m))m.
\end{equation}
\begin{lemma} \label{lem X phi}
One has
\[
X^{\HH}_1 = 2V_{\mu}.
\]
\end{lemma}
\begin{proof}
For all $m \in M$ and $v \in T_mM$, we compute
\[
\begin{split}
\omega_m(X^{\HH}_1(m), v) &=\langle (d_1\HH)_m, v\rangle \\	
	&= 2(T_m\mu(v), \mu(m))_m \\
	&= 2\langle d_m\mu_{\mu^*(m)}, v\rangle \\
	&= 2\omega_m(V_{\mu}(m), v).
\end{split}
\]
\end{proof}

Let $h_1, \ldots, h_{d_G}: M \to \kg^*$ be an orthonormal frame for the vector bundle $M \times \kg^* \to M$, with respect to the given family of inner products. Write
\begin{equation} \label{eq def muj}
\mu = \sum_{j=1}^{d_G} \mu_j h_j,
\end{equation}
for functions $\mu_j \in C^{\infty}(M)$. 
For each $j$, write 
\begin{equation} \label{eq def Vj}
V_j := V_{h_j}
\end{equation}
for the vector field induced by $h_j$ as in \eqref{eq Vf}. Then Lemma \ref{lem X phi} implies that
\begin{equation} \label{eq X1H frame}
\XoneH = 2 \sum_{j=1}^{d_G}\mu_jV_j.
\end{equation}
This is an analogue of (1.19) in \cite{TZ98}.

\subsection{Induced vector fields on reduced spaces} \label{sec vf red}

For a coadjoint orbit $\cO \subset \kg^*$, we denote the reduced space
at $\cO$ by  $M_{\cO} := \mu^{-1}(\cO)/G$. If $\cO$ consists of regular values of $\mu$, and $G$ acts freely on $\mu^{-1}(\cO)$, we denote the symplectic form on $M_{\cO}$ induced by $\omega$ as in  \cite{MW} by $\omega_{\cO}$. If $\xi \in \kg^*$, we write $M_{\xi} := M_{G\cdot \xi}$ and $\omega_{\xi} := \omega_{G\cdot \xi}$. 

As noted in Subsection \ref{sec X12H}, the vector fields $\XoneH$ and $X_2^{\HH}$ are $G$-invariant, and hence descend to $M/G$ at points with trivial stabilizers.
Because of Lemma \ref{lem X phi}, the vector field $\XoneH$ is tangent to $G$-orbits, so that it induces the zero vector field on the quotient.

The vector field $X_2^{\HH}$ is not necessarily tangent to orbits. It does have the weaker property that it is tangent to submanifolds of the form $\mu^{-1}(\cO)$, for  any coadjoint orbit $\cO \subset \kg^*$ that consists of regular values of $\mu$. This follows from the following fact.
\begin{lemma} \label{lem mu invar}
Let $f \in C^{\infty}(M)$ be a $G$-invariant function, and let $X^f$ be its Hamiltonian vector field. Then at every point $m \in M$,
\[
T_m\mu(X^f(m)) = 0 \quad \in \kg^*.
\]
\end{lemma}
\begin{proof}
For every $X \in \kg$, one has
\[
\langle T_m\mu(X^f_m), X\rangle = \langle d_m\mu_X, X^f_m\rangle 
	= \omega_m(X_m^M, X^f_m)
	= -\langle d_mf, X_m^M\rangle
	=0,
\]
since $f$ is $G$-invariant.
\end{proof}

\begin{corollary} \label{cor E tangent}
Let $m\in M$, and write $\xi := \mu(m)$ and $\cO := \Ad^*(G)\xi$. One has
\[
T_m\mu(X_2^{\HH}(m)) = -2\mu^*(m)_{\xi} \quad \in T_{\xi} \cO\hookrightarrow \kg^*.
\]
So if $\xi$ is a regular value of $\mu$, then
\[
X_2^{\HH}(m) \in T_m (\mu^{-1}(\cO)) = (T_m\mu)^{-1}(T_{\xi} \cO).
\]
\end{corollary}
\begin{proof}
Because of $G$-invariance of $\HH$, relation \eqref{eq decomp XH} and Lemmas  \ref{lem X phi} and 
  \ref{lem mu invar}, one has
\[
T_m\mu(X_2^{\HH}(m)) = -2T_m\mu(V_{\mu}(m))
\]
Equivariance of $\mu$ implies that the latter expression equals $-2\mu^*(m)_{\xi}$.
\end{proof}

Because $X_2^{\HH}$ is $G$-invariant and tangent to submanifolds of the form $\mu^{-1}(\cO)$ as above, it induces a vector field $(X_2^{\HH})_{\cO}$ on every symplectic reduction $M_{\cO}$, if $\cO$ consists of regular values of $\mu$, and $G$ acts freely on $\mu^{-1}(\cO)$.  As noted above, $X_1^{\HH}$ induces the zero vector field on the reduced spaces $M_{\cO}$.

We will mainly consider the case $\cO = \{0\}$.
\begin{lemma} \label{lem X2H0}
The  vector field $(X_2^{\HH})_{0}$ on $M_0$ induced by $X_2^{\HH}$ is the zero vector field. 
\end{lemma}
\begin{proof}
Let $\iota: \mu^{-1}(0) \hookrightarrow M$ be the inclusion map. Then $\iota^* (d\HH) = d(\iota^* \HH) = 0$, so $X^{\HH}$ induces the zero vector field on $M_0$.
Hence
\[
(X_2^{\HH})_{0} = (X_1^{\HH})_{0} + (X_2^{\HH})_{0} = (X^{\HH})_{0} = 0.
\]
\end{proof}

\subsection{Critical points} \label{sec crit}

For $j = 1,2$, let $\Crit_j(\HH)$ be the set of zeroes of $d_j\HH$, which equals the set of zeroes of $X^{\HH}_j$. We will later assume that $\Crit_1(\HH)$ is cocompact, and investigate that assumption in this subsection.

By Lemma \ref{lem X phi}, we have
\begin{equation} \label{eq crit stab}
\Crit_1(\HH) = \{m \in M; \mu^*(m) \in \kg_m\}.
\end{equation}
We will assume that $0$ is a regular value of $\mu$, which by Smale's lemma implies that $\kg_m = 0$ for all $m \in \mu^{-1}(0)$. Since the minimal isotropy type occurs on an open dense subset $U \subset M$ (see e.g.\ \cite{Michor}, Theorem 2.3), one has $\kg_m = 0$ for all $m \in U$. Therefore,
\begin{equation} \label{eq crit U}
\Crit_1(\HH) \cap U = \mu^{-1}(0).
\end{equation}
If $M_0 = \mu^{-1}(0)/G$ is compact, then \eqref{eq crit U} implies that any non-cocompact parts of $\Crit_1(\HH)$ are contained in the positive codimension part $M\setminus U$ of $M$. We have therefore established the following sufficient condition for compactness of $\Crit_1(\HH)/G$
\begin{lemma} \label{lem Crit1 cocpt}
Suppose that $0$ is a regular value of $\mu$, and that $M_0$ is compact.
Then  $\Crit_1(\HH)/G$ is compact if $(M\setminus U)/G$ is compact. 
\end{lemma}

%

An example where $\Crit_1(\HH)$ is cocompact is the action by a Lie group on its cotangent bundle.
\begin{example} \label{ex T*G crit}
Let $G$ be a Lie group, and consider the action by $G$ on its cotangent bundle $T^*G \cong G\times \kg^*$ induced by left multiplication. A momentum map for this action is the projection $\mu:G \times \kg^* \to \kg^*$, for which $\mu^{-1}(0) = G$. Since the action is free, and $\mu^{-1}(0)/G$ is a point, $\Crit_1(\HH)$ is cocompact by Lemma \ref{lem Crit1 cocpt}. This holds for any family of inner products on $\kg^*$.
\end{example}

Using a family of inner products on $\kg^*$ rather than a fixed inner product provides a flexibility that can be used to give the vector field $\XoneH$ some desirable properties. This may allow one to make $\Crit_1(\HH)$ more manageable, for example.
Specifically, it follows from \eqref{eq crit stab} that
\begin{equation} \label{eq incl crit}
\mu^{-1}(0) \cup M^G \subset \Crit_1(\HH),
\end{equation}
where $M^G$ denotes the fixed point set of the action. In certain cases, the converse inclusion holds as well.

Indeed, let $\Sym^+(\kg^*)$ be the set of positive definite symmetric linear automorphisms of $\kg^*$. 
Consider a smooth map
\[
b: M \to \Sym^+(\kg^*)
\]
which has the equivariance property that for all $m \in M$ and $g \in G$,
\[
b(gm) = \Ad^*(g) b(m) \Ad^*(g)^{-1}.
\]
Then setting
\[
(\relbar, \relbar)_m^b := \bigl( \relbar, b(m) \relbar \bigl)_m,
\]
for $m \in M$, defines a $G$-invariant metric $(\relbar ,\relbar)^b$ on $M \times \kg^* \to M$. Let $\HH^b$ be the resulting function defined as in \eqref{eq def H}. 

For $m \in M$, let $b^*(m) \in  \GL(\kg)$ be the linear endomorphism dual to $b(m)$. If the map $b$ can be chosen such that $b^*(m)\mu^*(m)$ points away from the stabilizer $\kg_m$ where this is possible, then the converse inclusion to \eqref{eq incl crit} holds as well.
\begin{lemma} \label{lem crit deform}
If for all $m \in M$ outside $\mu^{-1}(0) \cup M^G$, one has 
\begin{equation} \label{eq b mu}
b^*(m)\mu^*(m) \not\in \kg_m
\end{equation}
then
\[
\Crit_1(\HH^b) = \mu^{-1}(0) \cup M^G.
\]
\end{lemma}
\begin{proof}
Let 
\[
\mu^{*_b}: M \to \kg
\]
be the map dual to $\mu$, with respect to the family of inner products $\{(\relbar, \relbar)_m^b \bigr \}_{m \in M}$ on $\kg^*$, as defined in \eqref{eq dual f}. One can check that for all $m \in M$,
\[
\mu^{*_b}(m) = b^*(m) \mu^*(m).
\]
Hence it follows from \eqref{eq crit stab} that
\[
\Crit_1(\HH^b) = \{m \in M; b^*(m) \mu^*(m) \in \kg_m\}.
\]
The claim now follows from \eqref{eq b mu} and \eqref{eq incl crit}.
\end{proof}

Note that if $G$ is noncompact, properness of the action implies that $M^G$ is empty. Hence, in the situation of Lemma \ref{lem crit deform},  $\Crit_1(\HH)$ is cocompact if and only if $\mu^{-1}(0)$ is.

\subsection{Rescaling the metric} \label{sec rescale}

Another important flexibility one can exploit is rescaling a family of inner products on $\kg^*$ by a positive $G$-invariant function.
As before, let $\{(\relbar, \relbar)_m \bigr \}_{m \in M}$ be a $G$-invariant metric on $M \times \kg^* \to M$, and let  $\HH$ be the associated norm squared function \eqref{eq def H} of the momentum map. Let $\psi \in C^{\infty}(M)$ be a positive, $G$-invariant function, and let $\HH_{\psi}$ be the analogous function associated to the family of inner products $\{\psi(m)(\relbar, \relbar)_m \bigr \}_{m \in M}$:
\[
\HH_{\psi}(m) := \psi(m) \|\mu(m)\|_m^{2}.
\]
We will write $d_1\HH_{\psi}$ and $X_1^{\HH_{\psi}}$ for the one-form and vector field constructed from the function $\HH_{\psi}$ in the same way as $d_1\HH$ and $\XoneH$ were constructed from $\HH$.
\begin{lemma} \label{lem XoneH phi}
One has
\[
X_1^{\HH_{\psi}} = \psi \XoneH.
\]
\end{lemma}
\begin{proof}
Since computing the one-form $d_1\HH_{\psi}$ only involves differentiating with respect to the argument of $\mu$, one has
\[
d_1\HH_{\psi} = \psi d_1\HH.
\]
\end{proof}

\begin{remark}
As mentioned in Subsection \ref{sec crit}, the set $\Crit_1(\HH)$ will be assumed to be cocompact. By Lemma \ref{lem XoneH phi}, rescaling the metric by a positive, $G$-invariant function does not change the set $\Crit_1(\HH)$, and hence does not influence this assumption.
\end{remark}


\section{Assumptions and results} \label{sec deform}

In \cite{TZ98}, Tian and Zhang give an analytic proof that quantization commutes with reduction in cases where the manifold $M$ and the group $G$ are compact. Their proof is based on a Witten-type deformation 
\cite{Witten} of the $\Spin^c$-Dirac operator on a symplectic manifold. In this deformation, they use the Hamiltonian vector field $X^{\HH}$ of the norm-squared function $\HH$ of the momentum map. Crucially, the norm on $\kg^*$ used is invariant under the coadjoint action by $G$, which is always possible for compact groups.  In \cite{MZ}, Mathai and Zhang treat the cocompact case. Since an $\Ad^*(G)$-invariant inner product on $\kg^*$ is not always available then, they use a weighted average of the Hamiltonian vector field $X^{\HH}$, where $\HH$ is now defined with respect to a norm that is not necessarily $\Ad^*(G)$-invariant. 
We will use a different deformation of Dirac operators, using the vector field $X_1^{\HH}$ introduced in Section \ref{sec X1H}, instead of the vector field $X^{\HH}$. 

\subsection{Assumptions} \label{sec Dirac}

We make the same assumptions as for example in \cite{HL, MZ}, with the important exception that the orbit space of the action considered need not be compact. 

Let $(M,\om)$ be a symplectic
manifold. Let $J$ be an
almost complex structure on $M$ such that $\om(\relbar, J\relbar)$
 defines a Riemannian metric on $TM$.
 Assume $M$ is complete with respect to this metric. 

Assume that there exists a Hermitian line bundle
$L$ over $M$ carrying a Hermitian connection $\nabla^L$ such
that $\frac{\sqrt{-1}}{2\pi}(\nabla^L)^2=\omega$. 
Then for any $p \in \N$, the $p$'th tensor power $L^p \to M$ is a prequantum line bundle for the symplectic manifold $(M, p\omega)$.
 
For such an integer $p$,
Let $D^{L^p}:\Omega^{0,*}(M; L^p)\rightarrow \Omega^{0,*}(M; L^p)$
be the  $\Spin^c$-Dirac operator on $M$ coupled to the line bundle $L^p$ via the given prequantum data (see Section 1 of
\cite{TZ98}, or \cite{Duistermaat, Friedrich}). Let $D^{L^p}_+$ and $D^{L^p}_-$ be the restrictions of
$D^{L^p}$ to $\Omega^{0,\text{even}}(M; L^p)$ and $\Omega^{0,\text{odd}}(M; L^p)$, respectively.

Let $G$ be a unimodular Lie group, with Lie algebra $\kg$, acting properly and symplectically on $M$.
We assume  that the action of $G$ on $M$ lifts to $L$. Moreover,
we assume the $G$-action preserves the above metrics and
connections  on $TM$ and $L$, as well as the almost complex structure $J$. Then the operators 
$D_{\pm}^{L^p} $ commute with
the $G$-action.

The action of $G$ on $L$ naturally determines a momentum map $\mu:
M\to \kg^*$ such that for any $X\in \kg$ and $s\in\Gamma^{\infty}(L)$, if
$X^M$ denotes the induced Killing vector field on $M$, then the
following Kostant formula for the Lie derivative $L_{X^M}$ holds:
\begin{align}\label{1.0}
L_{X^M}s=\nabla_{X^M}^Ls-2\pi\sqrt{-1}\mu_X s.
\end{align}
For any integer $p$, and any section $s\in\Gamma^{\infty}(L^p)$, one then has
\[
L_{X^M}s=\nabla_{X^M}^{L^p}s-2\pi\sqrt{-1}p \mu_X s.
\]

We assume that  $0 \in \kg^*$ is a regular value
of $\mu$. It then follows from the definition of momentum maps that all stabilizers of the action by $G$ on $\mu^{-1}(0)$ are discrete, and hence finite by properness of the action. We will assume that these stabilizers are in fact trivial, i.e.\ that $G$ acts freely on $\mu^{-1}(0)$. Then the
Marsden--Weinstein symplectic reduction \cite{MW} $(M_0, \omega_0)$
is a smooth symplectic manifold. Moreover, the prequantum line bundle $L$ descends to a line bundle $L_{0}$ on $M_{0}$.  The connection $\nabla^L$ induces a connection $\nabla^{L_{0}}$ on $L_{0}$, such that
the corresponding curvature condition
$\frac{\sqrt{-1}}{2\pi}(\nabla^{L_{0}})^2=\omega_{0}$ holds. The
$G$-invariant almost complex structure $J$ also descends to an
almost complex structure $J_{0}$ on $M_{0}$, and the metrics on $L$ and $TM$ descend
to metrics on $L_{0}$ and $TM_{0}$, respectively. Let  $D^{L_{0}}$ denote the
corresponding $\Spin^c$-Dirac operator on $M_{0}$. These constructions generalize to yield the prequantum line bundle $L_0^p$ for the reduced space $(M_0, p\omega_0)$, and an associated Dirac operator $D^{L_0^p}$.

The manifold $M$ and the group $G$ are allowed to be noncompact independently, and there is no properness assumption on the momentum map $\mu$. The only compactness assumption made is that the set $\Crit_1(\HH)/G$ is compact for a $G$-invariant metric on $M\times \kg^* \to M$ (see Subsection \ref{sec crit}). By Lemma \ref{lem X phi}, this implies that the symplectic reduction $M_0$ is compact as well.


\subsection{Invariant quantization; the main results} \label{sec result}

Let $X_1^{\HH}$ be the vector field on $M$ introduced in Section \ref{sec X1H}, via a $G$-invariant metric on the trivial bundle $M \times \kg^* \to M$.
\begin{definition} \label{def deformed Dirac}
For $t \in \R$ and $p \in \N$, the \emph{deformed Dirac operator} on $\Omega^{0,*}(M; L^p)$ is the operator
\[
D^{L^p}_t := D^{L^p} + \frac{\sqrt{-1}t}{2}c(X_1^{\HH}),
\]
where $c$ denotes the Clifford action by $TM$ on $\mybigwedge^{0,*}T^*M$.
\end{definition}
The Clifford action $c$ by $TM$ on $\mybigwedge^{0,*}T^*M$ is explicitly defined as follows. Let $m \in M$, $v \in T_mM$, and let $v_{\C} = v^{1,0} + v^{0,1}$ be the decomposition of the complexification $v_{\C}$ of $v$ according to $T_mM \otimes \C = T^{1,0}_mM \oplus T^{0,1}_mM$. Let $(v^{1,0})^* \in (T^{0,1}_mM)^*$ be the covector dual to $v^{1,0}$ with respect to the metric. Then
\[
c(v) = \sqrt{2} \bigl( - i_{v^{0,1}} + (v^{1,0})^* \wedge \relbar \bigr): \mybigwedge^{0,*}T^*_mM \to \mybigwedge^{0,*}T^*_mM.
\]
Here $i_{v^{0,1}}$ denotes contraction by $v^{0,1}$.

Invariant quantization will be defined in terms of the \emph{transversally $L^2$-kernel} of $D^{L}_t$. The definition of this kernel involves the notion of a \emph{cutoff function}.

Let $dg$ be a left Haar measure on $G$. In e.g.\  \cite{Bourbaki}, Chapter VII, Section 2.4, Proposition 8, it is shown that a continuous, nonnegative function $f$ on $M$ exists, whose support intersects all $G$-orbits in compact sets,  and satisfies
\begin{equation} \label{eq int orbit}
\int_{G}f(g\cdot m)^2 \, dg = 1,
\end{equation}
for all $m \in M$. Such functions are used in many applications in index theory (see e.g.\ \cite{Kasparov2012, MZ}).
 If $M/G$ is compact, one may take $f$ to be compactly supported. We will call a function $f$ with these properties a \emph{cutoff function}. 

\begin{definition} \label{def transv L2}
Let $E\to M$ be a $G$-vector bundle, equpped with a $G$-invariant Hermitian metric.
The vector space of \emph{transversally $L^2$-sections} of $E$ is the space $L^2_T(E)$ of sections $s$ of $E$ (modulo equality almost everywhere) such that $fs \in L^2(E)$ for every cutoff function $f$. (Here integrals over $M$ are defined with respect to the Liouville measure.)

Let $D$ be a linear operator on $\Gamma^{\infty}(E)$. The \emph{transversally $L^2$-kernel} of  $D$ is the vector space
\[
\ker_{L^2_T}(D) := \{s \in \Gamma^{\infty}(E) \cap L^2_T(E); Ds = 0\}.
\]
\end{definition}

\begin{remark}
The vector space $L^2_T(E)$ can be given a locally convex topology via the seminorms
\[
\|s\|^f := \|fs\|_{L^2(E)},
\]
where $s \in L^2_T(E)$, and $f$ runs over the cutoff functions for the action by $G$ on $M$. (In fact, a set of cutoff functions whose supports cover $M$ is enough.)

On the $G$-invariant part\footnote{This also holds for subspaces of $L^2_T(E)$ with other kinds of transformation behavior under the action by $G$, as long as $||(g\cdot s)(m)|| = \|s(m)\|$ for all $g \in G$, $m \in M$ and $s$ in such a subspace.} $L^2_T(E)^G$ of $L^2_T(E)$, the expression
\[
(s, s')^f := (fs, fs')_{L^2(E)}
\]
defines an inner product, which is independent of the choice of $f$ (as shown in the proof of Lemma \ref{lem Sobolev indep f}). This turns $L^2_T(E)^G$ into a Hilbert space.
\end{remark}



The first main result of this paper is that the $G$-invariant part of the transversally $L^2$-kernel of the deformed Dirac operator is finite-dimensional for large $t$, so that it can  be used to define invariant quantization.
\begin{theorem}\label{thm quant well defd}
For $t$ large enough, the $G$-invariant part $\bigl(\ker_{L^2_T}(D^{L}_t) \bigr)^G$ of the vector space $\ker_{L^2_T}(D^{L}_t)$ is finite-dimensional. 
\end{theorem}
This result will be proved in Sections \ref{sec index}--\ref{sec localise}. As a consequence, invariant quantization can be defined as follows.
\begin{definition} \label{def quant}
The \emph{$G$-invariant geometric quantization} of the action by $G$ on $(M, \omega)$ is the integer
\begin{equation} \label{eq invar quant}
Q(M, \omega)^G := \dim \left(\ker_{L^2_T} \bigl( (D^{L}_t)_+ \bigr)  \right)^G - \dim \left(\ker_{L^2_T} \bigl( (D^{L}_t)_- \bigr)  \right)^G,
\end{equation}
for $t$ large enough.
\end{definition}

The second main result of this paper is that Conjecture \ref{con [Q,R]=0} is true for large enough powers of $L$.
\begin{theorem}[Quantization commutes with reduction for large $p$] \label{thm [Q,R]=0}
There is a $p_0 \in \N$, such that for all integers $p \geq p_0$,
\[
Q(M, p\omega)^G = Q(M_0, p\omega_0).
\]
\end{theorem}
Theorem \ref{thm [Q,R]=0} will be proved in Sections \ref{sec Bochner}--\ref{sec Dirac M M0}.

\begin{remark} \label{rem p = 1 free}
If $\Crit_1(\HH) = \mu^{-1}(0)$, which in particular occurs if the action is free, then one may in fact take $p_0 = 1$ in Theorem \ref{thm [Q,R]=0}. (See Remark \ref{rem p = 1}.)
\end{remark}

\begin{remark}
Let $p \in \N$. If one includes the symplectic form $\omega$ in the notation for $L = L_{\omega}$ and $\HH = \HH_{\omega}$, then one has
$L_{p\omega} = L_{\omega}^p$ and $\HH_{p \omega} = p^2 \HH_{\omega}$. Hence
\[
(p\omega)(X_1^{\HH_{p\omega}}, \relbar) = d_1 {\HH_{p\omega}} = p^2 {\HH_{\omega}}  = (p\omega)(pX_1^{\HH_{\omega}}, \relbar), 
\]
so $X_1^{\HH_{p\omega}} = p X_1^{\HH_{\omega}}$.
Note that the invariant quantization $Q(M, p\omega)^G$ is defined via the operator
\[
\begin{split}
D^{L_{p\omega}}_t &= D^{L_{p\omega}} + \frac{\sqrt{-1}t}{2}c(X_1^{\HH_{p\omega}}) \\
	&= D^{L_{\omega}^p} + \frac{\sqrt{-1}pt}{2}c(X_1^{\HH_{\omega}}) \\
	&= D^{L_{\omega}^p}_{p t}.
\end{split}
\]
Since one takes $t$ large enough in Definition \ref{def quant}, one may therefore also use the deformed Dirac operator $D^{L_{\omega}^p}_{t}$ of Definition \ref{def deformed Dirac} to define $Q(M, p\omega)^G$.
\end{remark}

\begin{remark}
A particular consequence of Theorem \ref{thm [Q,R]=0} is that, for $t$ and $p$ large enough, the integer \eqref{eq invar quant} is independent of the connection and Hermitian metric on $L$, the almost complex structure $J$ on $M$, the deformation parameter $t$, and the family of inner products on $\kg^*$ used (as long as they staisfy the assumptions listed). 
This can also be shown directly, 
as noted in Remark \ref{rem quant indep t}.
\end{remark}

\subsection{Special cases} \label{sec special}

Theorem \ref{thm [Q,R]=0} reduces to the main result in \cite{MZ} in the cocompact case. In the setting of the Vergne conjecture \cite{VergneICM} (a generalization of which was proved by Ma and Zhang \cite{Ma-Zhang2}), it yields information about the reduction at zero. An example where neither $M/G$ nor $G$ needs to be compact is the case where $M=T^*G$ is the cotangent bundle of $G$. Also, because $M/G$ is not assumed to be compact, a version of the shifting trick applies.

\begin{corollary} \label{cor MZ}
If $M/G$ is compact, Theorem \ref{thm [Q,R]=0} reduces to the case of Theorem 1.1 in \cite{MZ} where no $\Ad^*(G)$-invariant inner product on $\kg^*$ exists, and hence to the large $p$ case of Landsman's conjecture \eqref{eq Landsman}.
\end{corollary}
\begin{proof}
If $M/G$ is compact, then $\Crit_1(\HH)/G$ is always compact. 
Since all cutoff functions are compactly supported, all smooth sections are transversally $L^2$. 

To see that one may take $t=0$ in Definition \ref{def quant} in the cocompact case, one can use the index theory from Section \ref{sec index}.  If $M/G$ is compact, all $G$-invariant Sobolev norms of the same degree are equivalent. Hence the Sobolev spaces $W^k_f(M; L^p)^G_t$ defined in \eqref{eq Sob spinor} are equal to $W^k_f(M; L^p)^G_0$ for all $t$. These spaces in turn are equal to 
 the spaces $H^k_c(M, E)^G$ in the proof of Theorem 2.7 in \cite{MZ}.  These are isomorphic to the spaces $H^k_f(M, E)^G$ used there, via the map (2.22) in \cite{MZ}. 
 
By taking $V = M$ in Proposition \ref{prop Fredholm}, one sees that the deformed Dirac operator $\tilD^{L^p}_t$ is  Fredholm for any $t$ and $p$. Since the vector field $\XoneH$ has bounded norm, these deformed Dirac operators define a continuous family of Fredholm operators with respect to a single Sobolev norm, so that they have the same index.
Hence the index of $\tilD^{L^p}_t$ is independent of $t$, and  equal to the index of $\tilD^{L^p}$. 

By (4.3) in \cite{MZ}, in the cocompact case the index of $\tilD^{L^p}$ equals the index of the operator $P_f D^{L^p}$ used by Mathai and Zhang. Therefore, the invariant quantization $Q(M, p\omega)^G$ equals the left hand side of the equality in Theorem 1.1 in \cite{MZ}.

In the  to \cite{MZ}, Bunke shows that Theorem 1.1 in \cite{MZ} implies the large $p$ case of Landsman's conjecture. In fact, Conjecture \ref{con [Q,R]=0} reduces to Landsman's conjecture in the cocompact case.
\end{proof}

\begin{remark}
The case of Theorem 1.1 in \cite{MZ} where $\kg^*$ admits an $\Ad^*(G)$-invariant inner product and $p = 1$ is not a direct consequence of Theorem \ref{thm [Q,R]=0}, but is closely related. Indeed, in that setting, the constant metric on $M\times \kg^* \to M$ defined by this inner product has the properties in Proposition \ref{prop choice family}, if one takes $V = M$. As noted in \cite{MZ}, the techniques from \cite{TZ98} generalize directly to that case.
\end{remark}

\begin{corollary} 
Consider the setting of the Vergne conjecture \cite{VergneICM}, where $G$ and $\Crit(\HH)$ are compact\footnote{As noted in Lemma 3.24 in \cite{Paradan3},
the set $\Crit(\HH)$ is compact if $M$ is real-algebraic and $\mu$ is algebraic and proper.} and $\mu$ is proper. If $G$ acts freely on $\mu^{-1}(0)$, then for $t$ and $p$ large enough, one has
\[
\dim \left(\ker_{L^2} \bigl( (D^{L^p}_t)_+ \bigr)  \right)^G - \dim \left(\ker_{L^2} \bigl( (D^{L^p}_t)_- \bigr) \right)^G
= Q(M_0, p\omega_0).
\]
with $\ker_{L^2}\bigl((D^{L^p}_t)_{\pm} \bigr)$ the spaces of $L^2$-sections in the kernels of the 
 the even and odd parts $(D^{L^p}_t)_{\pm}$ of $D^{L^p}_{t}$, respectively.
\end{corollary}
\begin{proof}
Since an $\Ad^*(G)$-invariant inner product on $\kg^*$ exists if $G$ is compact, one can use a constant family of inner products on $\kg^*$. Then one has $\Crit_1(\HH) = \Crit(\HH)$. Hence compactness of $\Crit_1(\HH)$ is equivalent to  (co)compactness of $\Crit(\HH)$. 

If $G$ is compact, then $f \equiv 1$ is a cutoff function. Hence $G$-invariant sections are transversally $L^2$ precisely if they are $L^2$. Therefore,
\[
Q(M, p\omega)^G = \dim \left(\ker_{L^2} \bigl( (D^{L^p}_t)_+ \bigr)  \right)^G - \dim \left(\ker_{L^2} \bigl( (D^{L^p}_t)_- \bigr) \right)^G,
\]
for $p$ and $t$ large enough.
\end{proof}

Let $T^*G$ be the cotangent bundle of $G$, equipped with the standard symplectic form $\omega$. Consider the action by $G$ on $T^*G$ induced by left multiplication (see also \cite{Duistermaat}, p.\ 197). 
\begin{corollary} \label{cor cotangent}
One has
\[
Q(T^*G, \omega)^G = 1.
\]
\end{corollary}
\begin{proof}
It was noted in Example \ref{ex T*G crit} that $\Crit_1(\HH) = \mu^{-1}(0) = G$  in this case. So $\Crit_1(\HH)/G$ is a point, hence compact. Furthermore, $G$ acts freely on $\mu^{-1}(0)$. Since $(T^*G)_0$ is a point, its quantization equals $1$. By Remark \ref{rem p = 1 free},  Theorem \ref{thm [Q,R]=0} applies with $p_0 = 1$, so that
$Q(T^*G, \omega)^G = 1$.
\end{proof}

\begin{remark}
Intuitively, one would expect the geometric quantization of $T^*G$ to be the space $L^2(G)$. This was made precise for compact $G$ by Paradan \cite{Paradan2}. With Definition \ref{def quant} of invariant quantization, one would expect the invariant quantization of $T^*G$ to be the space of $G$-invariant functions on $G$ that are square integrable after multiplication by a compactly supported function. This is the one-dimensional space of constant functions on $G$, in accordance with Corollary \ref{cor cotangent}.
\end{remark}

Because of the cocompactness assumption in \cite{HL, Landsman, MZ}, it was impossible to apply any version of the shifting trick. Indeed, even if $M/G$ is compact, the diagonal action by $G$ on $M \times \cO$ will not be, if $\cO$ is a noncompact coadjoint orbit. In the present setting, the shifting trick does apply to a certain extent.

Let $\cO = \Ad^*(G)\xi$ be a strongly elliptic coadjoint orbit of $G$, with Kirillov--Kostant--Souriau symplectic form $\omega_{\cO}$.   Let $\HH_{\cO}$ be the function $\HH$ as in \eqref{eq def H}, for the diagonal action by $G$ on $\bigl(M\times \cO, \omega\times( -\omega_{\cO}) \bigr)$. 
\begin{corollary}[Shifting trick]
Suppose that $\cO$ consists of regular values of $\mu$, and $G$ acts freely on $\mu^{-1}(\cO)$. If $\Crit_1(\HH_{\cO})$ is cocompact, then for large enough $p$, 
\[
Q(M_{\xi}, p\omega_{\xi}) = Q \bigl(M\times\cO, p(\omega \times -\omega_{\cO}) \bigr)^G.
\]
\end{corollary}
\begin{proof}
Because $\cO$ is a strongly elliptic orbit, the action by $G$ on $\cO$ is proper. Hence Theorem \ref{thm [Q,R]=0} applies. By the standard shifting trick, the symplectic reduction  of $\bigl(M\times\cO, p(\omega \times -\omega_{\cO}) \bigr)$ at zero equals $(M_{\xi}, p\omega_{\xi})$, and the claim follows.
\end{proof}


\section{A $G$-invariant index} \label{sec index}


This section contains some general index theory, which will be used to prove that invariant quantization is well-defined (Theorem \ref{thm quant well defd}). It wil be shown that $G$-invariant quantization is the index of a Fredholm operator between Sobolev spaces defined in this section. Elliptic operators satisfying an invertibility condition outside a cocompact set will be shown to define such Fredholm operators. In Subsection \ref{sec loc est}, we will see that the deformed Dirac operator $D^{L^p}_t$ satisfies this condition, for $t$ large enough. This will complete the proof of Theorem \ref{thm quant well defd}.

\subsection{Sobolev spaces} \label{sec sobolev}

Let $M$ be a smooth manifold, and let $G$ be a unimodular Lie group acting properly on $M$.  
Let $dg$ be a Haar measure on $G$. Let $f$ be a smooth cutoff function for the action by $G$ on $M$ (see Subsection \ref{sec result}). 
Suppose that $M$ is equipped with 
a $G$-invariant Borel measure $dm$ (which will later be assumed to be given by the density associated to a $G$-invariant Riemannian metric).
Let $E \to M$ be a $G$-vector bundle, equipped with a  $G$-invariant Hermitian metric $(\relbar, \relbar)_E$. We denote the space of smooth sections of $E$ by $\Gamma^{\infty}(E)$, and the space of smooth, compactly supported sections of $E$ by $\Gamma^{\infty}_c(E)$.

We will use Sobolev spaces constructed from spaces of \emph{transversally compactly supported} sections of vector bundles. 
\begin{definition}
The \emph{transversal support} of a section $s$ of $E$ is the closure in $M/G$ of the set of orbits $\cO \in M/G$ such that there is a point $m \in \cO$ where $s(m) \not= 0$. If the transversal support of a section $s$ is compact, then $s$ is called \emph{transversally compactly supported}. 
\end{definition}
The space of smooth, transversally compactly supported sections of $E$ will be denoted by $\Gamma^{\infty}_{tc}(E)$. The space of $G$-invariant sections in $\Gamma^{\infty}_{tc}(E)$ is denoted by $\Gamma^{\infty}_{tc}(E)^G$.
For all $s \in \Gamma^{\infty}_{tc}(E)$, the product $fs$ is a smooth, compactly supported section of $E$. 

Let
\[
D: \Gamma^{\infty}(E) \to \Gamma^{\infty}(E)
\]
be a first order, $G$-equivariant, essentially self-adjoint, elliptic differential operator. 
We will write $\sigma_D$ for the principal symbol of $D$. Suppose that $E$ is $\Z/2$ graded, and write $E = E_+ \oplus E_-$ for the decomposition induced by this grading. Suppose that $D$ is odd with respect to the grading.
Define the operator
\[
\tilD: f \Gamma^{\infty}_{tc}(E)^G \to f \Gamma^{\infty}_{tc}(E)^G
\]
by
\begin{equation} \label{eq til D}
\widetilde{D}(fs) = fDs
\end{equation}
if $s \in \Gamma^{\infty}_{tc}(E)^G$. 

Using the measure $dm$ on $M$ and the Hermitian metric $(\relbar, \relbar)_E$ on $E$, one can define an $L^2$-inner product on compactly supported smooth sections of $E$. The operator $\tilD$ is not symmetric with respect to this inner product in general. This fact
 will play a role in the proof of Proposition \ref{prop adjoint D}.

Consider the $k$'th Sobolev inner product on $f \Gamma^{\infty}_{tc}(E)^G$ defined by
\begin{equation} \label{eq Sob norms}
(fs, fs')_{W^k_f(E)^G} := \sum_{j=0}^k \bigl(\tilD^j(fs), \tilD^j(fs')\bigr)_{L^2(E)},
\end{equation}
for $s, s' \in \Gamma^{\infty}_{tc}(E)^G$. 
\begin{definition} \label{def Hkf}
The Sobolev space $W^k_f(E)^G$ is the completion of $f\Gamma^{\infty}_{tc}(E)^G$ in the inner product defined by \eqref{eq Sob norms}.
\end{definition}
\begin{definition}  \label{def tilde D}
The bounded operator
\[
\widetilde{D}: W^1_f(E)^G \to W^0_f(E)^G.
\]
is the continuous extension of \eqref{eq til D}.
\end{definition}


\subsection{Properties of the spaces $W^k_f(E)^G$} \label{sec dep f D}

%



If $G$ is unimodular, choosing different cutoff functions $f$ to define the Sobolev spaces $W^k_f(E)^G$  leads to canonically isomorphic spaces. 
To prove this, we will use the measure\footnote{Existence and uniqueness of this measure is proved for any proper action in \cite{Bourbaki}, Chapter VII, Section 2.2, Proposition 4b.
 It can be constructed as follows. Let  $h_G \in C_c^{\infty}(M/G)$, and let $f$ be a cutoff function for the action. Then set
\[
\int_{M/G} h_G(\cO)\, d\cO:= \int_M f(m)^2 h_G(Gm)\, dm.
\]
A version of this measure for non-unimodular groups $G$ also exists, but then the invariance condition on $dm$ is replaced by a condition involving the modular function.
} 
$d\cO$ on $M/G$ such that for all $h \in C_c(M)$, 
\begin{equation} \label{eq measure quotient}
\int_M h(m) \, dm = \int_{M/G} \int_G h(g \tau(\cO)) \, dg\, d\cO,
\end{equation}
for any Borel section $\tau: M/G \to M$. 
\begin{lemma} \label{lem Sobolev indep f}
Suppose $G$ is unimodular.
Let $f_0$ and $f_1$ be two cutoff functions for the action by $G$ on $M$. Then the map given by
\begin{equation} \label{eq map indep f}
f_0 s \mapsto f_1 s
\end{equation}
for $s \in \Gamma^{\infty}_{tc}(E)^G$ induces a unitary isomorphism $W^k_{f_0}(E)^G \cong W^k_{f_1}(E)^G$.
\end{lemma}
\begin{proof}
The map \eqref{eq map indep f} is a bijection from $f_0\Gamma^{\infty}_{tc}(E)^G$ to $f_1\Gamma^{\infty}_{tc}(E)^G$, because $G$-invariant sections are determined by their restrictions to $\supp(f_0)$ or $\supp(f_1)$.

For $k = 0$, one finds for all $s, s' \in \Gamma^{\infty}_{tc}(E)^G$,
\[
\begin{split}
(f_j s, f_j s')_{W^0_{f_j}(E)^G} &= \int_M f_j(m)^2\bigl(s(m), s'(m) \bigr)_{E} \, dm \\
	&= \int_{M/G} \int_G f_j(g\tau(\cO))^2  \bigl(s(g\tau(\cO)), s'(g\tau(\cO)) \bigr)_{E} \, dg \, d\cO.
\end{split}
\]
By the property \eqref{eq int orbit} of cutoff functions, and $G$-invariance of $s$ and $s'$, the latter expression equals
\[
\int_{M/G}  \bigl(s(\tau(\cO)), s'(\tau(\cO)) \bigr)_{E} \, d\cO,
\]
which is independent of $j$.

For general $k$, one notes that for all $s, s' \in \Gamma^{\infty}_{tc}(E)^G$,
\[
\bigl(\tilD^k f_js, \tilD^k f_js' \bigr)_{L^2(E)} =  \bigl(f_j D^k s, f_jD^ks' \bigr)_{L^2(E)},
\]
which by the argument above is independent of $j$.
%
This implies that
\[
(f_0s, f_0s')_{W^k_{f_0}(E)^G} = (f_1s, f_1s')_{W^k_{f_1}(E)^G}
\]
for all $k$, as required.
\end{proof}
Since we assumed  that $G$ is unimodular,  the Sobolev spaces $W^k_f(E)^G$ are indeed independent of $f$.

An analogue of the Rellich lemma holds for the restricted Sobolev spaces $W^{k}_f(E|_V)^G$, if $V \subset M$ is a $G$-invariant, relatively cocompact open subset of $M$.
\begin{lemma} \label{lem cpt incl}
Let $V \subset M$ be $G$-invariant, relatively cocompact and open. Then for all $k \geq 0$, the inclusion map
\[
W^{k+1}_f(E|_V)^G \hookrightarrow W^{k}_f(E|_V)^G
\]
is compact.
\end{lemma}
\begin{proof}
Let $Z$ be the intersection of $V$ with the interior of $\supp(f)$. Restricting to $Z$ is an injective map
\[
W^k_f(E|_V)^G \hookrightarrow W^k(E|_Z).
\]
This map is an isometry if $W^k(E|_Z)$ is defined with respect to the Sobolev inner product coming from the elliptic operator $D - \sigma_D(df)$. Since $Z$ is relatively compact, all $k$'th Sobolev norms on sections of $E|_V$ are equivalent, and one may as well use the one defined by $D - \sigma_D(df)$.

Consider the diagram
\[
\xymatrix{
W^{k+1}(E|_Z) \ar@{^{(}->}[r] & W^k(E|_Z) \\
W^{k+1}_f(E|_V)^G \ar@{^{(}->}[r]  \ar@{^{(}->}[u]  & W^k_f(E|_V)^G. \ar@{^{(}->}[u] 
}
\] 
We have seen that the vertical inclusion maps are isometric, and the Rellich lemma implies that the top inclusion map is compact. Hence the bottom inclusion is compact was well, as required.
\end{proof}

\subsection{Free actions} \label{sec free}

Suppose for now that the action by $G$ on $M$ is \emph{free}, in addition to being proper. 
(The material in this subsection will be used in Section \ref{sec Dirac M M0}, where $M$ is replaced by a neighborhood of $\mu^{-1}(0)$.)
For free actions, one has the induced vector bundle
\[
E_G \to M/G,
\]
such that $q^*E_G \cong E$, with $q: M\to M/G$ the quotient map. 
The Hermitian metric on $E$ induces one on $E_G$.
The induced operator $D^G$ on sections of $E_G$,
\[
D^G: \Gamma^{\infty}(E_G) \to \Gamma^{\infty}(F_G),
\]
is again essentially self-adjoint and elliptic. Using this measure and this metric, and the operator $D^G$, one can define Sobolev spaces $W^k(E_G)$ of sections of $E_G$. These are the completions of $\Gamma^{\infty}_c(E_G)$ in the 
inner product given by
\[
(s, s')_{W^k(E_G)} := \sum_{j=0}^k \bigl((D^G)^j s, (D^G)^js'\bigr)_{L^2(E_G)},
\]
for $s, s' \in \Gamma^{\infty}_c(E_G)$. Then the operator $D^G$ extends to a bounded operator
\[
D^G: W^1(E_G) \to W^0(E_G).
\]

\begin{lemma} \label{lem iso sobolev}
The composition
\[
\Gamma^{\infty}_{c}(E_G) \xrightarrow{q^*} \Gamma^{\infty}_{tc}(E)^G \xrightarrow{f\cdot} f\Gamma^{\infty}_{tc}(E)^G
\]
extends to a unitary isomorphism 
\[
\psi: W^k(E_G) \xrightarrow{\cong} W^k_f(E)^G.
\]
\end{lemma}
\begin{proof}
We only need to check that multiplication by $f$ composed with $q^*$ is an isometry. The argument is analogous to the proof of Lemma \ref{lem Sobolev indep f}. In the same way as there, one first computes that
for all $s, s' \in \Gamma^{\infty}_{c}(E_G)$,
\[
(fq^*s, fq^*s')_{W^0_f(E)^G} = (s, s')_{W^0(E_G)}.
\]
For general $k$, one then notes that for all such $s, s'$,
\[
(\tilD^k fq^*s, \tilD^k fq^*s')_{L^2(E)} =  \bigl(fq^*(D^G)^k s, fq^*(D^G)^k s'\bigr)_{L^2(E)} = \bigl( (D^G)^k s, (D^G)^k s' \bigr)_{L^2(E_G)}. 
\]
This implies that
\[
(fq^*s, fq^*s')_{W^k_f(E)^G} = (s, s')_{W^k(E_G)},
\]
for all $k$, as required.
\end{proof}

The operators $D$, $\widetilde{D}$ and $D^G$ are related by the commutative diagram
\[
\xymatrix{
\Gamma^{\infty}_{c}(E_G) \ar[r]^-{q^*} \ar[d]^-{D^G}& \Gamma^{\infty}_{tc}(E)^G \ar[r]^-{f\cdot} \ar[d]^-{D}& f\Gamma^{\infty}_{tc}(E)^G \ar[d]^-{\widetilde{D}} \\
\Gamma^{\infty}_{c}(F_G) \ar[r]^-{q^*} & \Gamma^{\infty}_{tc}(F)^G \ar[r]^-{f\cdot} & f\Gamma^{\infty}_{tc}(F)^G.
}
\]
In other words, the unitary isomorphism $\psi$ from Lemma \ref{lem iso sobolev} intertwines the operators $D^G$ and $\widetilde{D}$.

\subsection{The Fredholm property}

In \cite{Anghel}, Anghel gives a criterion for an elliptic, self-adjoint differential operator on a noncompact manifold to be Fredholm: an $L^2$-norm estimate outside a compact subset of the manifold. This generalizes Gromov and Lawson's results for Dirac operators in Section 3 of \cite{GL}. In our setting, where one considers operators between the Sobolev spaces $W^1_f(E)^G$ and $W^0_f(E)^G$, the analogous estimate outside a cocompact subset is sufficient. 
\begin{proposition} \label{prop Fredholm}
Suppose $M$ carries a $G$-invariant Riemannian metric. Let the measure $dm$ be given by the Riemannian density, and suppose $M$ is complete.  
Let $K \subset M$ be a cocompact subset, and suppose there is a $C>0$ such that for all $s \in \Gamma^{\infty}_{tc}(E)^G$ with support disjoint from $K$, 
\begin{equation} \label{eq cond Fredholm}
\| \widetilde{D}fs \|_{L^2(E)} \geq C \|fs\|_{L^2(E)}.
\end{equation}
Then the operator $\widetilde{D}: W^1_f(E)^G \to W^0_f(E)^G$ is Fredholm.
%
\end{proposition}
\begin{proof}
Let $(s_j)_{j=1}^{\infty}$ be a sequence in $\Gtc(E)^G$, such that
\begin{itemize}
\item the sequence $(fs_j)_{j = 1}^{\infty}$ is bounded in $W^0_f(E)^G$;
\item the sequence $(\tilD fs_j)_{j=1}^{\infty}$ converges in $W^0_f(E)^G$.
\end{itemize}
It is enough to prove that there is a subsequence $(s_{j_k})_{k=1}^{\infty}$ such that $(fs_{j_k})_{k=1}^{\infty}$ converges in $W^0_f(E)^G$, i.e.\ in $L^2(E)$. This implies that $\tilD$ has finite-dimensional kernel and closed range;  see for example the last two paragraphs of the proof of Theorem 2.1 in \cite{Anghel}.

Let $U$ be a relatively cocompact open neighborhood of $K$. Let $Z$ be the intersection of $U$ with the interior of the support of $f$. Then $Z$ is relatively compact. Hence all Sobolev norms of the same degree on $\Gamma^{\infty}_c(E)$ have equivalent restrictions to $\Gamma^{\infty}_c(E|_Z)$. The sequence $(fs_j)_{j=1}^{\infty}$ is bounded in $W^1_f(E)^G$, so its restriction to $Z$ is bounded with respect to any first Sobolev norm. Therefore, the Rellich lemma yields
a subsequence $(s_{j_k})_{k=1}^{\infty}$ such that $(fs_{j_k}|_Z)$ converges in $L^2(E|_Z)$. 

Because $f|_U$ is zero outside $Z$, the sequence $(fs_{j_k}|_U)$ converges in $L^2(E|_U)$, and hence in $W^0_f(E|_U)^G$. 
Let $\chi \in C^{\infty}(M)^G$ be a smooth, $G$-invariant  function, with values in $[0,1]$, such that
\begin{itemize}
\item $\chi \equiv 1$ on $K$;
\item $\chi \equiv 0$ on $M\setminus \overline{U}$.
\end{itemize}
Then the sequence $(\chi fs_{j_k})_{k=1}^{\infty}$ converges in $W^0_f(E)^G$.

To show that the sequence  $\bigl((1-\chi) fs_{j_k}\bigr)_{k=1}^{\infty}$ converges in $W^0_f(E)^G$,
one first notes that
 the commutator $[D, \chi]$ equals $\sigma_D(d\chi)$, with $\sigma_D$ the principal symbol of $D$. Hence for all $s \in \Gtc(E)^G$,
\[
[\tilD, \chi]fs = f[D,\chi] s = \sigma_D(d\chi)fs.
\]
Since $\sigma_D(d\chi)$ is $G$-equivariant, and zero outside the relatively cocompact set $U\setminus K$, the operator
\[
\sigma_D(d\chi): W^0_f(E)^G \to W^0_f(E)^G
\]
is bounded.  Because $\chi$ is supported in $\overline{U}$, 
convergence of $(fs_{j_k}|_U)_{k=1}^{\infty}$ and boundedness of $\sigma_D(d\chi)$ imply convergence of $([\tilD, \chi] fs_{j_k})_{k=1}^{\infty}$.

Because $1-\chi$ is supported outside $K$, the assumption \eqref{eq cond Fredholm} implies that for all $k, l$,
\begin{multline*}
C\|(1-\chi)f(s_{j_k} - s_{j_l})\|_{W^0_f(E)^G} \leq \bigl\|\tilD\bigl(  (1-\chi)f(s_{j_k} - s_{j_l}) \bigr)  \bigr\|_{W^0_f(E)^G} \\
	= \bigl\| (1-\chi)\tilD\bigl( f(s_{j_k} - s_{j_l}) \bigr)  -  [\tilD, \chi]  f(s_{j_k} - s_{j_l})  \bigr\|_{W^0_f(E)^G} \\
	\leq \bigl\| \tilD\bigl( f(s_{j_k} - s_{j_l}) \bigr) \|_{W^0_f(E)^G}+ \| [\tilD, \chi]  f(s_{j_k} - s_{j_l})  \|_{W^0_f(E)^G}. 
\end{multline*}
Both terms in the latter expression become arbitrarily small, since $(\tilD fs_j)_{j=1}^{\infty}$ and $([\tilD, \chi] fs_{j_k})_{k=1}^{\infty}$ converge. Hence the sequence $\bigl((1-\chi)fs_{j_k}\bigr)_{k=1}^{\infty}$ converges. 

Since $(\chi fs_{j_k})_{k=1}^{\infty}$ converges as well, we conclude that $(fs_{j_k})_{k=1}^{\infty}$ converges in $W^0_f(E)^G$, as required.
\end{proof}

\subsection{The $G$-invariant index}

Suppose the conditions  of Proposition \ref{prop Fredholm} are satisfied. Then the Fredholm operator $\tilD$ has a well-defined index. The proof of Theorem \ref{thm quant well defd} is based on the following explicit description of this index, in terms of the transversally $L^2$-kernel of $D$ (see Definition \ref{def transv L2}).  
\begin{proposition} \label{prop L2t index}
In the situation of Proposition \ref{prop Fredholm}, the $G$-invariant part $\bigl( \ker_{L^2_T}(D) \bigr)^G$ of the transversally $L^2$-kernel of $D$ is finite-dimensional, and one has
\begin{equation} \label{eq invar index}
\ind(\tilD) = \dim \bigl( \ker_{L^2_T}(D_+) \bigr)^G -  \dim \bigl( \ker_{L^2_T}(D_-) \bigr)^G.
\end{equation}
\end{proposition}
A proof of this fact is given in  Appendix \ref{app reg L2t}.

\begin{definition} \label{def indG}
In the setting of Proposition \ref{prop Fredholm}, the \emph{$G$-invariant index} of $D$ is the number \eqref{eq invar index}.
It is denoted by $\ind_G(D)$.
\end{definition}

In Subsection \ref{sec loc est}, it is shown that the deformed Dirac operator $D^{L^p}_t$ satisfies the hypotheses of Proposition \ref{prop Fredholm}, for $t$ large enough. Hence invariant quantization is well-defined as its $G$-invariant index, and Theorem \ref{thm quant well defd} follows. 
Note that this is different from the cocompact situation considered in \cite{MZ}, where the \emph{undeformed} Dirac operators were already Fredholm on suitable Sobolev spaces.


\section{The square of the deformed Dirac operator} \label{sec Bochner}

As in \cite{MZ, TZ98}, the square of the deformed Dirac operator $D^{L^p}_t$ plays an important role. An expression for this square is given in Theorem \ref{thm Bochner}, which generalizes Theorem 1.6 in \cite{TZ98}. As in  Corollary 1.7 in \cite{TZ98}, a  term involving Lie derivatives vanishes on $G$-invariant sections, and one is left with the square of $D^{L^p}$ plus order zero terms. This is recorded in Corollary \ref{cor Bochner invar}. 

In the compact and cocompact cases, the zero order terms in the square of the deformed Dirac operator are automatically bounded. This is not true for non-cocompact manifolds.  In Subsection \ref{sec choice family}, it will be shown that, for a well-chosen family of inner products on $\kg^*$, these zero order terms satisfy an estimate that can be used to localize the invariant index of $D^{L^p}_t$ to neighborhoods of $\Crit_1(\HH)$ and $\mu^{-1}(0)$.

\subsection{A Bochner formula}

We first introduce some operators and vector fields that will be used in the expression for the square of $D^{L^p}_t$. Recall the notation of Subsection \ref{sec X12H}.
In particular, we chose an orthonormal frame $\{h_1, \ldots, h_{d_G}\}$ for the bundle $M \times \kg^* \to M$, with respect to the given $G$-invariant metric on this bundle. Let $\{h_1^*, \ldots, h_{d_G}^*\}$ be the dual frame of $M\times \kg \to M$ as in \eqref{eq dual f}. For each $j$, consider the operator $L_{h^*_j}$ on $\Omega^{0,*}(M; L^p)$ given by
\[
(L_{h_j^*}s)(m) = (L_{h^*_j(m)}s)(m),
\]
for all $s \in \Omega^{0,*}(M; L^p)$ and $m \in M$. We will use the fact that $L_{h^*_j}$ annihilates $G$-invariant sections. 

In addition, for any vector field $v \in \XX(M)$, consider the commutator vector field $[v, (h^*_j)^M ]$, given by
\[
[v, (h^*_j)^M](m) = \bigl[v, h^*_j(m)^M \bigr](m).
\]
Here $h^*_j(m)^M$ is the vector field induced by $h^*_j(m) \in \kg$, and $[\relbar, \relbar]$ is the Lie bracket of vector fields. Importantly, for fixed $m$, the vector fields $V_j$ in \eqref{eq def Vj} and $h^*_j(m)^M$ are equal at the point $m$, but not necessarily at other points.

Finally, the one-form $\langle \mu, Th_j^* \rangle \in \Omega^1(M)$ is defined by
\[
\bigl\langle \langle \mu, Th_j^* \rangle_m, v \bigr\rangle =  \langle \mu(m), T_mh_j^*(v) \rangle
\]
for all $m \in M$ and $v \in T_mM$. The dual vector field associated to $\langle \mu, Th_j^* \rangle$ via the Riemannian metric will be denoted by $\langle \mu, Th_j^* \rangle^*$.

\begin{theorem} \label{thm Bochner}
Let $p \in \N$ and $t \in \R$ be given. The square of the deformed Dirac operator $D^{L^p}_t$ equals
\begin{equation} \label{eq Bochner}
\bigl( D^{L^p}_t\bigr)^2 = \bigl( D^{L^p}\bigr)^2 + tA + 4\pi p t \HH + \frac{t^2}{4} \|\XoneH\|^2 -2\sqrt{-1} t \sum_{j=1}^{d_G} \mu_j L_{h_j^*}, 
\end{equation}
where $A$ is a vector bundle endomorphism of $\bigwedge^{0,*}T^*M \otimes L^p$, equal to $A = A_1 + A_2 + A_3$, with the endomorphisms $A_n$ defined as follows. Let $e_1, \ldots, e_{d_M}$ be a local orthonormal frame for $TM$, and write $e_k^{1,0}$ for the component of the complexification of $e_k$ in $T^{1,0}M$. 
Let $\nabla$ be the Levi--Civita connection on $TM$, and $\nabla^{T^{1,0}M}$ the induced connection on $T^{1,0}M$.
The endomorphism $A_1$ is the Tian--Zhang tensor that appears in Theorem 1.6 in \cite{TZ98}:
\begin{multline}
A_1 = \frac{\sqrt{-1}}{4} \Sk c(e_k)c\left(\nabla_{e_k} X_1^{\HH} \right) 
	-\frac{\sqrt{-1}}{2}  \tr\left(  \left. \nabla^{T^{1,0}M} \XoneH \right|_{T^{0,1}M}  \right) \label{eq def A1}\\
	+\frac{1}{2} \sum_{j=1}^{d_G} \left( \sqrt{-1} c(JV_j)c(V_j) + \|V_j\|^2\right). 
\end{multline}
The tensors $A_2$ and $A_3$ appear because the orthonormal frame $\{h_1, \ldots, h_{d_G}\}$ of $M \times \kg^* \to M$ may not be constant:
\begin{equation} \label{eq def A2}
A_2 = \frac{\sqrt{-1}}{2} \sum_{j=1}^{d_G} c(\langle \mu, Th_j^* \rangle^*) c(V_j) 
	+ \sqrt{-1} \sum_{j=1}^{d_G} \sum_{k = 1}^{d_M} \bigl\langle \langle \mu, Th_j^* \rangle_m, e_k^{1,0} \bigr\rangle \bigl(V_j, e_k^{1,0} \bigr);
\end{equation}
and 
\begin{multline}
A_3 = -\frac{\sqrt{-1}}{2} \sum_{j=1}^{d_G} \sum_{k = 1}^{d_M} \mu_j c(e_k) c\bigl(  \bigl[e_k, (h^*_j)^M - V_j\bigr] \bigr) \\
-\sqrt{-1} \Sjk \mu_j \bigl(  \bigl[e_k^{1,0}, (h^*_j)^M - V_j\bigr], e_k^{1,0} \bigr).   \label{eq def A3} 
\end{multline}
\end{theorem}

Since the operators $L_{h_j^*}$ in \eqref{eq Bochner} map  $G$-invariant sections to zero, one obtains the following analogue of Corollary 1.7 in \cite{TZ98}.
\begin{corollary} \label{cor Bochner invar}
For $p \in \N$ and $t \in \R$ , the restriction of $\bigl( D^{L^p}_t\bigr)^2$ to $\Omega^{0,*}(M; L)^G$ equals
\begin{equation} \label{eq Bochner invar}
 \bigl( D^{L^p}\bigr)^2 + tA + 4\pi p t \HH + \frac{t^2}{4} \|\XoneH\|^2.
\end{equation}
\end{corollary}

\begin{remark}
If there is an $\Ad^*(G)$-invariant inner product on $\kg^*$, then Theorem \ref{thm Bochner} and Corollary \ref{cor Bochner invar} imply Theorem 1.6 and Corollary 1.7 in \cite{TZ98}, respectively. Indeed, in that case one may use a constant family of inner products on $\kg^*$, and a constant orthonormal frame for $M \times \kg^* \to M$. Then $\XoneH = X^{\HH}$. Since $h_j^*$ is constant, 
one has $T_mh_j^* = 0$, so $\langle \mu, Th_j^* \rangle = 0$. Also, 
one has $h^*_j(m)^M = V_j$ for all $m \in M$, so that $ \bigl[v, (h^*_j)^M - V_j\bigr] = 0$ for any vector field $v$. 
Therefore,  $A_2 = A_3 = 0$, and the remaining part of \eqref{eq Bochner} becomes the equality in Theorem 1.6 in \cite{TZ98}. 

Note that, even if $\kg^*$ admits an $\Ad^*(G)$-invariant inner product, one may wish to use a non-constant family of inner products, for example to apply Lemma \ref{lem crit deform} or Proposition \ref{prop choice family}. Then one should use Theorem \ref{thm Bochner} and Corollary \ref{cor Bochner invar} rather than Theorem 1.6 and Corollary 1.7 in \cite{TZ98}.
\end{remark}

\subsection{Proof of the Bochner formula}

We give an outline of a proof of Theorem \ref{thm Bochner} here, and provide more details in  Appendix \ref{app Bochner}.

\noindent \emph{Proof of Theorem \ref{thm Bochner}}.
As in the equality (1.26) in \cite{TZ98}, one finds that
\begin{equation} \label{eq Bochner 0}
\bigl( D^{L^p}_t\bigr)^2 = \bigl( D^{L^p}\bigr)^2 + \frac{\sqrt{-1}}{2}t \Sk c(e_k) c(\nabla_{e_k}\XoneH)
	-\sqrt{-1}t \nabla_{\XoneH} + \frac{t^2}{4} \|\XoneH\|^2.
\end{equation}
The bulk of the proof of Theorem \ref{thm Bochner} consists of computing an expression for the first order term $\nabla_{\XoneH}$. Because of \eqref{eq X1H frame}, one has
\[
\nabla_{\XoneH} = 2 \Sj \mu_j \nabla_{V_j}. 
\]
For all $m \in M$, one has $V_j(m) = h^*_j(m)^M_m$. Since covariant derivatives $\nabla_v$ depend locally on the vector field $v$,  Lemma 1.5 in \cite{TZ98} therefore implies that for all $s \in \Omega^{0,*}(M; L)$ and all $m \in M$,
\begin{multline} \label{eq nabla Vj}
(\nabla_{V_j} s)(m)= \bigl(\nabla_{h^*_j(m)^M}s \bigr)(m) = (L_{h_j^*(m)}s)(m) + 2\pi \sqrt{-1} p \mu_j(m) s(m)  \\
 + \frac{1}{4}\Sk \bigl(c(e_k) c(\nabla_{e_k} h^*_j(m)^M)\bigr)(m)s(m) \\
 + \frac{1}{2} \tr \left( \nabla^{T^{1,0}M} h^*_j(m)^M|_{T^{1,0}M} \right)(m)s(m).
\end{multline}
Multiplying \eqref{eq nabla Vj} by $2\mu_j$ and summing over $j$ yields
\begin{multline} \label{eq Bochner 1}
(\nabla_{\XoneH} s)(m) = 
	2 \Sj  \mu_j(m) (L_{h_j^*(m)}s)(m) + 4\pi \sqrt{-1} p \HH(m) s(m) \\ 
	+ \frac{1}{2} \Sj \Sk \mu_j(m)  \Bigl(c(e_k) c(\nabla_{e_k} h^*_j(m)^M) \Bigr)(m)s(m) \\
 + \Sj \mu_j(m) \tr \left( \nabla^{T^{1,0}M} h^*_j(m)^M|_{T^{1,0}M} \right)(m) s(m).
\end{multline}

In Proposition \ref{prop nabla X1H} in  Appendix \ref{app Bochner}, it is deduced from \eqref{eq Bochner 1} that
\begin{multline} \label{eq Bochner 6}
(\nabla_{\XoneH} s)(m) = 
	2 \Sj  \mu_j(m) (L_{h_j^*(m)}s)(m) + 4\pi \sqrt{-1} p \HH(m) s(m) \\
	+ \left( \frac{1}{4} \Sk c(e_k) c(\nabla_{e_k} \XoneH)
	+ \frac{1}{2}\Sj \bigl( -c(JV_j)c(V_j) + \sqrt{-1} \|V_j\|^2 \bigr)  \right. \\
	\left.
	 + \frac{1}{2} \tr \left( \nabla^{T^{1,0}M}\XoneH|_{T^{1,0}M}   \right) \right)(m) s(m)\\
+\sqrt{-1}(A_2 + A_3)(m) s(m),	
\end{multline}
with $A_2$ and $A_3$ as in \eqref{eq def A2} and \eqref{eq def A3}.

The equality \eqref{eq Bochner} follows by inserting \eqref{eq Bochner 6} into \eqref{eq Bochner 0}, and using the definition \eqref{eq def A1} of the operator $A_1$.
\hfill $\square$


\section{Localizing the deformed Dirac operator} \label{sec localise}

The key steps in the proofs of the main results in this paper, Theorems \ref{thm quant well defd} and \ref{thm [Q,R]=0}, is showing that the kernel of $D^{L^p}_t$ localizes to neighborhoods of $\Crit_1(\HH)$ and $\mu^{-1}(0)$, respectively. The localization estimates used to do this are Propositions \ref{prop loc V} and \ref{thm loc global}, which will be proved in this section. Recall that the set $\Crit_1(\HH)/G$,  and hence the set $\mu^{-1}(0)/G$, was assumed to be compact.

\subsection{Localization estimates} \label{sec loc est}

Let $\Omega^{0,*}_{tc}(M; L^p)$ be the space of smooth sections of the vector bundle $\bigwedge^{0,*}T^*M \otimes L^p \to M$ with  compact transversal supports. Consider the Sobolev inner product \eqref{eq Sob norms} on $f\Omega^{0,*}_{tc}(M; L^p)$, defined with respect to the operator $D = D^{L^p}_t$ on $E = \bigwedge^{0,*}T^*M \otimes L^p$. We write $(\relbar, \relbar)_{k, t}$ for the resulting $k$'th Sobolev inner product, and $\|\cdot\|_{k, t}$ for the induced norm. In addition, we write $(\relbar, \relbar)_{k} := (\relbar, \relbar)_{k, 0}$ and $\|\cdot\|_{k} := \|\cdot\|_{k, 0}$. Note that $(\relbar, \relbar)_{0, t} = (\relbar, \relbar)_{0}$ is the $L^2$ inner product for all $t$.

Let 
\begin{equation} \label{eq Sob spinor}
W^k_f(M; L^p)^G_t := W_f^k\bigl(\mybigwedge^{0,*}T^*M\otimes L^p\bigr)^G,
\end{equation}
be the completion of $f\Omega^{0,*}_{tc}(M; {L^p})^G$ in the inner product $(\relbar, \relbar)_{k, t}$. Then the deformed Dirac operator $D^{L^p}_t$ induces a bounded operator
\[
\tilD^{L^p}_t: W^1_f(M; {L^p})^G_t \to W^0_f(M; {L^p})^G_t.
\]

The first localization estimate will be used to prove Theorem \ref{thm quant well defd}, which states that invariant quantization is well defined.
\begin{proposition} \label{prop loc V}
There is a family of inner products on $\kg^*$, a relatively
cocompact open neighborhood $V$ of $\Crit_1(\HH)$, and there are constants $C>0$, $b > 0$ and $t_0>0$, such that for every $t \geq t_0$ and all $s \in \Omega^{0,*}_{tc}(M; {L})^G$ with support disjoint from $V$,
\[
\| \tilD_t^{L} (fs) \|^2_0 \geq C\bigl(\|fs\|_1^2 +  (t-b)\|fs\|^2_0 \bigr).
\]
\end{proposition}

\begin{remark} \label{rem quant indep t}
Based on Proposition \ref{prop loc V}, one expects Definition \ref{def quant} to be independent of the connection and metric on $L$, the almost complex structure on $M$, the deformation parameter $t$, and the family of inner products on $\kg^*$. Indeed, it means that the index of $\tilD^{L^p}_t$ is determined on the relatively cocompact set $V$. On $V$, the norm of $\XoneH$ is bounded, and all $G$-invariant Sobolev norms are equivalent. Hence the family of operators
$\bigl( \tilD^{L^p}_t|_V\bigr)_{t \in \R}$ depends continuously on $t$, with respect to a fixed Sobolev norm. 

Furthermore,  the deformed Dirac operators defined by two families of metrics on $\kg^*$ have bounded difference on $V$, with respect to the $L^2$-norm. Hence, by Lemma \ref{lem cpt incl}, this difference is a compact operator between the first and zero'th Sobolev spaces in question.

The intuitive idea was made rigorous by Braverman \cite{Braverman2}, who extended the $G$-invariant, transversally $L^2$-index to  more general Dirac-type operators, and proved it is invariant under a suitable notion of cobordism. This shows that Definition \ref{def quant} is indeed independent of the choices made. 
\end{remark}

The second localization estimate is a further localization to neighborhoods of $\mu^{-1}(0) \subset \Crit_1(\HH)$, where one uses tensor powers of the line bundle $L$.
\begin{proposition} \label{thm loc global}
There is an equivariant family of inner products on $\kg^*$ such that for any $G$-invariant\footnote{In \cite{TZ98}, the analogous estimate is proved for any open neighborhood of $\mu^{-1}(0)$, not necessarily $G$-invariant. If $\mu^{-1}(0)$ is compact, any neighborhood contains a $G$-invariant one. In the noncompact case, one has to assume $G$-invariance. This is no restrictive assumption, however.} open neighborhood $U$ of $\mu^{-1}(0)$, there are $p_0 \in \N$, $t_0>0$ and $C>0$ and $b>0$, such that for all integers $p \geq p_0$, real numbers $t\geq t_0$ and all $s \in \Omega^{0,*}_{tc}(M; {L^p})^G$ with support disjoint from $U$,
\[
\| \tilD_t^{L^p} (fs) \|^2_0 \geq C\bigl( \|fs\|^2_1 + (t-b)\|fs\|^2_0 \bigr).
\]
\end{proposition}
Importantly,  one may use the same family of inner products on $\kg^*$ in Propositions \ref{prop loc V} and \ref{thm loc global}, namely the family constructed in Proposition \ref{prop choice family}.

\begin{remark}
Proposition \ref{thm loc global} directly implies that quantization commutes with reduction if $0$ is not a value of $\mu$ at all. Indeed, 
one can then take $U=\emptyset$ in Proposition \ref{thm loc global}, and conclude that $\tilD^{L^p}_t$ has trivial kernel for $t > b$ and large $p$. Therefore, by Proposition \ref{prop L2t index},
\[
Q(M, p\omega)^G = 0
\]
in such cases. From now on, we suppose that  $0 \in \mu(M)$.
\end{remark}

\begin{remark} \label{rem p = 1}
If $\Crit_1(\HH) = \mu^{-1}(0)$, then Proposition \ref{prop loc V} implies that Proposition \ref{thm loc global} holds with $p_0 = 1$. The number $p_0$ in Proposition \ref{thm loc global} is the same as the number $p_0$ in Theorem \ref{thm [Q,R]=0}, so the latter result also holds for $p_0 = 1$ in this case.
\end{remark}

Before proving Propositions \ref{prop loc V} and \ref{thm loc global}, we show how Propositions \ref{prop loc V} and \ref{prop Fredholm} imply Theorem \ref{thm quant well defd}.

\noindent \emph{Proof of Theorem \ref{thm quant well defd}.}
 Let $C, b$ and $t_0$, as well as a family of inner products on $\kg^*$ and a set $V$ as in Proposition \ref{prop loc V} be given.
For $t$ larger than both $t_0$ and $b+1$ and for $s \in \Omega^{0,*}_{tc}(M; {L})^G$ with support disjoint from $V$, one then has
\[
\| \tilD_t^{L} (fs) \|^2_0 \geq C\bigl( \|fs\|^2_1 + (t-b)\|fs\|^2_0 \bigr) \geq C\|fs\|_0^2.
\]
Proposition \ref{prop Fredholm} therefore implies that $\tilD^{L}_t$ is Fredholm. 
Because of Proposition \ref{prop L2t index}, the vector space $\bigl(\ker_{L^2_T}(D^L_t) \bigr)^G$ is finite-dimensional, and the integer
\[
\dim \left(\ker_{L^2_T} \bigl( (D^{L}_t)_+ \bigr)  \right)^G - \dim \left(\ker_{L^2_T} \bigl( (D^{L}_t)_- \bigr) \right)^G
\]
is the Fredholm index of $\tilD^{L}_t$. 
\hfill $\square$


\subsection{Choosing the family of inner products on $\kg^*$} \label{sec choice family}

One of the main difficulties in generalizing Tian and Zhang's localization theorem to the non-cocompact setting, is that the operator $A$ in \eqref{eq Bochner invar} may not be bounded below. We overcome this difficulty by rescaling the family of inner products on $\kg^*$ by a $G$-invariant positive function $\psi$ on $M$, as discussed in Subsection \ref{sec rescale}.
Because of Lemma \ref{lem XoneH phi}, rescaling the inner products in this way results in replacing the vector field $\XoneH$ by $\psi \XoneH$. A version of this technique, of deforming a Dirac operator by Clifford multiplication by a vector field, and then rescaling this vector field by a function, was used by Braverman \cite{Braverman} in the case where $G$ is compact. 

Fix a relatively cocompact open neighborhood $V$ of $\Crit_1(\HH)$. This is where cocompactness of $\Crit_1(\HH)$ is used. We will also use the $G$-invariant, positive smooth function $\eta$ on $M$ defined by
\begin{equation} \label{eq eta}
\eta(m) =\int_G f(gm)\|df\|(gm) \, dg,
\end{equation}
for $m \in M$.
\begin{proposition} \label{prop choice family}
The $G$-invariant metric on the bundle $M \times \kg^* \to M$ can be chosen in such a way that 
 for all $m \in M \setminus V$,
\begin{align}
\HH(m) &\geq 1; \label{eq est H} \\
\|\XoneH(m)\| &\geq 1 + \eta(m), \label{eq est X1H}
\end{align}
and 
there is a positive constant $C$, such that
for all $m \in M$, the operator $A_m$ on ${\bigwedge}^{0,*}T^*_mM \otimes L_m^p$ is bounded below by
\begin{equation} \label{eq A bdd below}
A_m \geq -C(\|\XoneH(m)\|^2 + 1).
\end{equation}
\end{proposition}
\begin{proof}
We outline a proof here, and refer to  Appendix \ref{app bound A} for certain details.

By Lemma \ref{lem cover} in  Appendix \ref{app bound A}, there is an open cover $\{\widetilde{U}_l\}_l$ of $M$ such that 
\begin{itemize}
\item every open set $\widetilde{U}_l$ admits a local orthonormal frame for $TM$;
\item every compact subset of $M$ intersects finitely many of the sets $\widetilde{U}_l$ nontrivially;
\item there is a relatively compact subset $U_l \subset \widetilde{U}_l$ for all $l$, such that $\overline{U_l} \subset \widetilde{U}_l$, and  $\bigcup_l U_l = M$.
\end{itemize}
Fix a local orthonormal frame $\{e_1^{l}, \ldots, e_{d_M}^l\}$ for $TM$ on every set $\widetilde{U}_l$. For all $k$ and $l$, let $(e_k^l)^{1,0}$ be the component of the complexification of $e_k^l$ in $T^{1,0}M$.
Let $W \subset M$ be a subset whose intersection with every nonempty cocompact subset of $M$ is nonempty and compact. 

Consider any $G$-invariant metric on $M \times \kg^* \to M$,  let $\HH$ be the associated function \eqref{eq def H}, and let $\XoneH$ be the vector field defined by \eqref{eq def XjH}. Fix an orthonormal frame $\{h_1, \ldots, h_{d_M}\}$ of $M \times \kg^* \to M$, and let $\{h_1^*, \ldots, h_{d_M}^*\}$ be the dual frame of $M \times \kg \to M$ as in \eqref{eq dual f}.

By Lemma \ref{lem F}, there are positive, $G$-invariant, continuous functions $F_1, F_2, F_3 \in C(M)^G$ such that for all $m \in W$, and for all $l$ such that $m \in U_l$, and all  $j = 1, \ldots, d_G$ and $k= 1, \ldots, d_M$,
\[
\begin{split}
\| \nabla_{e_k^l}\XoneH \|(m) &\leq F_1(m); \\
\| \nabla^{T^{1,0}M}_{(e_k^l)^{1,0}}\XoneH \|(m) &\leq F_1(m); \\
\|\langle \mu, Th_j^*\rangle \|(m) &\leq F_2(m); \\
\bigl\|  \bigl[(e_k^l)^{1,0}, (h^*_j)^M - V_j\bigr]  \bigr\|(m) &\leq F_3(m); \\
\bigl\|  \bigl[e_k^l, (h^*_j)^M - V_j\bigr]  \bigr\|(m) &\leq F_3(m).
\end{split}
\]
Let $N := \|\XoneH\|^2$ denote the norm-squared function of $\XoneH$. Define the continuous, nonnegative, $G$-invariant functions $\varphi_0$ and $\varphi_1$ on $M$ by\footnote{If $f_1$ and $f_2$ are functions, and $f_2$ is nonnegative, then by $\min(f_1, 1/f_2)$ we mean the function equal to $\min(f_1, 1/f_2)$ where $f_2$ is nonzero, and to $f_1$ where $f_2$ is zero.}
\[
\begin{split}
\varphi_0 &= \min\left(\HH, \frac{N^{1/2}}{1+\eta}, \frac{N}{F_1}, \min_j \frac{N}{\|V_j\| F_2 },   \min_j \frac{N}{|\mu_j| F_3}\right);\\
\varphi_1 &= \min\left(N^{1/2}, 
\min_j \frac{2N}{|\mu_j| \|V_j\|}
\right) 
\end{split}
\]
Note that the functions $\varphi_0$ and $\varphi_1$ are strictly positive  outside $\Crit_1(\HH)$. Therefore, by Lemma \ref{lem bdd der app}, there is a $G$-invariant, positive, smooth function $\psi$ on $M$, such that on $M\setminus V$,
\begin{align}
\psi^{-1} &\leq \varphi_0; \label{eq bound psi}\\
\| d(\psi^{-1}) \| &\leq \varphi_1. \label{eq bound dpsi}
\end{align}

As in Subsection \ref{sec rescale}, consider the family of inner products $\{ \psi(m) (\relbar, \relbar)_m \}_{m \in M}$ on $\kg^*$.
This is again a $G$-invariant smooth metric on $M \times \kg^* \to M$. We will show that this metric has the desired properties.
As in Lemma \ref{lem XoneH phi}, let $\HH_{\psi}$ be the corresponding norm-squared function of the momentum map $\mu$, so that $X_1^{\HH_{\psi}} = \psi \XoneH$. Write $N_{\psi} := \|X_1^{\HH_{\psi}}\|^2 = \psi^2 N$.

First of all, note that, outside $V$,  one has
\[
\begin{split}
\HH_{\psi} &= \psi \HH  \geq \varphi_0^{-1} \HH \geq 1; \\
\|X_1^{\HH_{\psi}} \| &= \psi \|\XoneH \| \geq \varphi_0^{-1} \|\XoneH \| \geq 1 +\eta,
\end{split}
\]
so \eqref{eq est H} and \eqref{eq est X1H} follow. 

We now turn to a proof of the lower bound \eqref{eq A bdd below} for the operator $A = A_1+  A_2 + A_3$. We will find a bound for each of the operators $A_n$ separately.
Write $A_n^{\psi}$ for the operators in \eqref{eq def A1}--\eqref{eq def A3}, with $\HH$ replaced by $\HH_{\psi}$. 

We will use the orthonormal frame of $M \times \kg^* \to M$ made up of the functions
\[
h_j^{\psi} := \frac{1}{\psi^{1/2}}h_j.
\]
The dual frame of $M \times \kg \to M$ consists of the functions
\[
(h_j^{\psi})^* =  \psi^{1/2} h_j^*.
\]
Let $\mu_j^{\psi}$ be defined like the functions $\mu_j$ in \eqref{eq def muj}, with $h_j$ replaced by $h_j^{\psi}$. Analogously, let $V_j^{\psi}$ be the vector field defined like $V_j$ in \eqref{eq def Vj}, with the same replacement. Then
\[
\begin{split}
\mu_j^{\psi} &= {\psi^{1/2}} \mu_j; \\
V_j^{\psi} &= \psi^{1/2}V_j.
\end{split}
\]


First, consider the operator $A_1^{\psi}$. 
On each of the sets $U_l$, we use the local orthonormal frame $\{e_k := e_k^l\}_{k=1}^{d_M}$ for $TM$, and set
\begin{align}
\widetilde{A}_1^{\psi} &:= \frac{\sqrt{-1}}{4} \Sk c(e_k)c\left(\nabla_{e_k} X_1^{\HH_{\psi}} \right) 
	-\frac{\sqrt{-1}}{2}  \tr\left(  \left. \nabla^{T^{1,0}M} X_1^{\HH_{\psi}} \right|_{T^{0,1}M}  \right)   \label{eq A tilde}\\
	&= A_1^{\psi}- \frac{1}{2} \sum_{j=1}^{d_G} \left( \sqrt{-1} c(JV_j^{\psi})c(V_j^{\psi}) + \|V_j^{\psi}\|^2\right). \nonumber
\end{align}

By Lemma \ref{lem nabla X1H}, one has
\[
\|\nabla_{v} X_1^{\HH_{\psi}} \| \leq 2N_{\psi},
\]
on $(M \setminus V) \cap W$, if $v$ is one of the vector fields $e_k$ or $e_k^{1,0}$. By Lemma \ref{lem bound Apsi}, this implies that, on that set, the operator $\widetilde{A}_1^{\psi}$ satisfies the pointwise estimate
\[
\| \widetilde{A}_1^{\psi}\| \leq \frac{3}{2}d_M N_{\psi}.
\]
Therefore,
\[
\Realpart (\widetilde{A}_1^{\psi}) \geq -\frac{3}{2} d_M N_{\psi},
\]
on $(M \setminus V) \cap W$. In addition, one has
\[
\sqrt{-1}c(JV_j^{\psi})c(V_j^{\psi}) + \|V_j^{\psi}\|^2 \geq 0
\]
for all $j$ (see e.g.\ (2.13) in \cite{TZ98}).
Therefore, on  $(M\setminus V) \cap W$, we obtain
\begin{equation} \label{eq bound A1}
\Realpart (A_1^{\psi})  = \Realpart (\widetilde{A}_1^{\psi}) + \frac{1}{2} \sum_{j=1}^{d_G} \left( \sqrt{-1} c(JV_j^{\psi})c(V_j^{\psi}) + \|V_j\|^2\right)  \geq -\frac{3}{2} d_M N_{\psi}.
\end{equation}


To estimate the norm of the operator $A_2^{\psi}$, we use the equality in Lemma \ref{lem mu Thj psi}:
\[
\langle \mu, T(h_j^{\psi})^* \rangle = \mu_j d(\psi^{1/2}) + \psi^{1/2} \langle \mu, Th_j^*\rangle. 
\]
By Lemma \ref{lem est A2}, this implies that, for all $j$, one has
\[
\| \langle \mu, T(h_j^{\psi})^* \rangle \| \cdot  \|V_j^{\psi}\| \leq 2N_{\psi}
\]
on $(M\setminus V) \cap W$. This allows one to find a bound for $\|A_2^{\psi}\|$. Indeed, one has
\begin{equation} \label{eq bound A2}
\begin{split}
\|A_2^{\psi}\| &\leq \frac{1}{2} \Sj \| \langle \mu, T(h_j^{\psi})^* \rangle \| \cdot  \|V_j^{\psi}\| + \Sjk \| \langle \mu, T(h_j^{\psi})^* \rangle \| \cdot  \|V_j^{\psi}\| \\
	&\leq \left(d_G + 2 d_G d_M \right) N_{\psi}.
\end{split}
\end{equation}


To estimate the norm of the operator $A_3^{\psi}$, we use Lemma \ref{lem Lie bracket psi}, which states that for all vector fields $v \in \XX(M)$ and all $j$,
\[
 \bigl[v, \bigl((h^*_j)^{\psi}\bigr)^M - V_j^{\psi}\bigr] = \psi^{1/2}  \bigl[v, (h^*_j)^M - V_j\bigr] - v(\psi^{1/2}) V_j.
\]
Let $v$ be one of the vector fields $e_k$ or $e_k^{1,0}$.
Then by Lemma \ref{lem est A3}, one has for all $j$, on $(M\setminus V) \cap W$,
\[
|\mu_j^{\psi}|   \bigl\| \bigl[v, \bigl((h^*_j)^{\psi}\bigr)^M - V_j^{\psi}\bigr]  \bigr\|  \leq 2 N_{\psi}.
\]
It then follows from the definition of the operator $A_3^{\psi}$ that, on $(M\setminus V) \cap W$,
\begin{equation} \label{eq bound A3}
\begin{split}
\|A_3^{\psi}\| &\leq \Sjk |\mu_j^{\psi}| \bigl\| \bigl[e_k^{1,0}, \bigl((h^*_j)^{\psi}\bigr)^M - V_j^{\psi}\bigr]  \bigr\|
	+ \frac{1}{2} \Sjk  |\mu_j^{\psi}| \bigl\| \bigl[e_k, \bigl((h^*_j)^{\psi}\bigr)^M - V_j^{\psi}\bigr]  \bigr\| \\
	&\leq 3d_G d_M N_{\psi}.
\end{split}
\end{equation}

Because of \eqref{eq bound A1}, \eqref{eq bound A2} and \eqref{eq bound A3}, there is a constant $C' > 0$ such that, on $(M \setminus V) \cap W$,
\[
A^{\psi} \geq -C'N_{\psi}.
\]
Since both sides of this inequality are $G$-invariant, it holds on all of $M\setminus V$. Finally, because $\overline{V}/G$ is compact, $A^{\psi}$ is bounded below on $V$. Therefore, the estimate \eqref{eq A bdd below} follows on all of $M$.
\end{proof}

From now on, we suppose the family of inner products on $\kg^*$ has the properties in Proposition \ref{prop choice family}, and we omit the function $\psi$ from the notation.

\subsection{Estimates for adjoint operators}

As noted in Subsection \ref{sec sobolev}, the operator $\tilD^{L^p}_t$ is not symmetric with respect to the $L^2$-inner product in general. Let $(\tilD^{L^p}_t)^*$ be the operator on $f\Omega^{0,*}_{tc}(M; L^p)$ such that for all $s, s' \in \Omega^{0,*}_{tc}(M; L^p)$
\[
\bigl((\tilD^{L^p}_t)^* fs, fs' \bigr)_{0} = \bigl( fs, \tilD^{L^p}_t fs'\bigr)_0.
\]
In particular, for $t=0$, one has the operator $(\tilD^{L^p})^* = (\tilD^{L^p}_0)^*$. 
Theorem \ref{thm Bochner} and the estimates in Proposition \ref{prop choice family} turn out to imply the following property of the operator $(\tilD^{L^p}_t)^* \tilD^{L^p}_t$.
\begin{proposition} \label{prop adjoint D}
One has
\[
(\tilD^{L^p}_t)^* \tilD^{L^p}_t = (\tilD^{L^p})^* \tilD^{L^p} + tB + 4\pi p t \HH + \frac{t^2}{4} \|\XoneH\|^2,
\]
for an operator $B$ 
for which there is a constant $C' > 0$ such that for all $m \in M$,
\[
B_m \geq -C'(\|\XoneH(m)\|^2 + 1).
\]
\end{proposition}

To prove Proposition \ref{prop adjoint D}, we use the following expression for $(\tilD^{L^p}_t)^*$.
\begin{lemma} \label{lem adjoint Dt}
For all $s \in \Omega^{0,*}_{tc}(M; L^p)$, one has
\[
(\tilD^{L^p}_t)^* fs = \tilD^{L^p}_t fs + 2c(df)s.
\]
\end{lemma}
\begin{proof}
For $s, s' \in \Omega^{0,*}_{tc}(M; L^p)$, we compute
\begin{equation} \label{eq adjoint 1}
\bigl((\tilD^{L^p}_t)^* fs, fs' \bigr)_{0} = \bigl( fs, f\tilD^{L^p}_t s'\bigr)_0 
	= \bigl( fs, D^{L^p}_t fs'\bigr)_0 - \bigl( fs, c(df) s'\bigr)_0.
\end{equation}
By symmetry of the operator $D^{L^p}_t$, the first term on the right hand side of \eqref{eq adjoint 1} equals
\[
\bigl(D^{L^p}_t fs,  fs'\bigr)_0 = \bigl( \tilD^{L^p}_t fs + c(df)s,  fs'\bigr)_0.
\]
By antisymmetry of $c(df)$, the second term on the right hand side of \eqref{eq adjoint 1} equals
\[
- \bigl( fs, c(df) s'\bigr)_0 =  \bigl(c(df) fs,  s'\bigr)_0,
\]
and the claim follows.
\end{proof}

\begin{lemma} \label{lem Dt*Dt}
For all $s \in \Omega^{0,*}_{tc}(M; L^p)$, one has
\begin{equation} \label{eq adjoint 2}
\bigl((\tilD^{L^p}_t)^* \tilD^{L^p}_t \bigr) fs = \Bigl((\tilD^{L^p})^* \tilD^{L^p} + tA + 4\pi p t \HH + \frac{t^2}{4} \|\XoneH\|^2 \Bigr)fs + \sqrt{-1}tc(df)c(\XoneH)s,
\end{equation}
where $A$ is the operator from Theorem \ref{thm Bochner}. 
\end{lemma}
\begin{proof}
Because of Lemma \ref{lem adjoint Dt}, one has for all $s \in \Omega^{0,*}_{tc}(M; L^p)$,
\[
(\tilD^{L^p}_t)^* \tilD^{L^p}_t fs   =   (\tilD^{L^p}_t)^2 fs + 2c(df)D^{L^p}_t s.
\]
Subtracting this equality for $t=0$ from the equality for general $t$, one gets
\[
\bigl((\tilD^{L^p}_t)^* \tilD^{L^p}_t  -  (\tilD^{L^p})^* \tilD^{L^p}\bigr)   fs   =   \bigl( (\tilD^{L^p}_t)^2  - (\tilD^{L^p})^2\bigr) fs + itc(df)c(\XoneH)s.
\]
The desired equality therefore follows from Theorem \ref{thm Bochner}.
\end{proof}

A priori, it is not clear if the operator $fs \mapsto \sqrt{-1} c(df)c(\XoneH)s$ that appears in  \eqref{eq adjoint 2} can be bounded in a suitable way. One has the following estimate, however.
\begin{lemma} \label{lem est L2 clifford}
There is a constant $C'' > 0$ such that for all $s \in \Omega^{0,*}_{tc}(M; L^p)$,
\[
\Realpart \bigl( \sqrt{-1}c(df)c(\XoneH)s, fs \bigr)_0 \geq -C'' \bigl(  (\|\XoneH\|^2+1)fs, fs \bigr)_0.
\]
\end{lemma}
\begin{proof}
Let $s \in \Omega^{0,*}_{tc}(M; L^p)$. Then, using the measure  $d\cO$ from \eqref{eq measure quotient}, we find that
\begin{multline}
\bigl| \bigl(\sqrt{-1}c(df)c(\XoneH)s, fs \bigr)_0 \bigr| \leq \int_M \bigl(f  \|df\| \cdot  \|\XoneH\| \cdot \|s\|^2  \bigr)(m)\, dm 
\\
	\leq \int_{M/G} \left( \int_G f(g\tau(\cO)) \|df\|(g\tau(\cO))  \, dg\right) \|\XoneH (\tau(\cO))\| \cdot  \|s(\tau(\cO))\|^2\, d\cO 
	 \\
	= \int_{M/G}\eta(\tau(\cO)) \|\XoneH (\tau(\cO))\| \cdot  \|s(\tau(\cO))\|^2\, d\cO.
	\label{eq adjoint 5}
\end{multline}
Here we have used $G$-invariance of the functions $\|\XoneH\|$ and $\|s\|$, and 
$\eta$ is the function defined by \eqref{eq eta}.

Since we use the metric on $M \times \kg^* \to M$ of Proposition \ref{prop choice family}, the function $\eta$ satisfies $\eta \leq \|\XoneH\|$ outside the set $V$. And since $\eta$ is $G$-invariant, it is bounded on the set $\overline{V}$. Hence there is a $C'''>0$ such that
\[
\eta \leq C''' \bigl( \|\XoneH\| + 1\bigr)
\]
on all of $M$. Let $C'' > 0$ be such that 
\[
C''' \bigl( \|\XoneH\| + 1\bigr) \|\XoneH\| \leq C'' \bigl( \|\XoneH\|^2 + 1\bigr). 
\]
Then \eqref{eq adjoint 5} is at most equal to
\[
C'' \int_{M/G} \bigl( \|\XoneH(\tau(\cO))\|^2 + 1\bigr) \|s(\tau(\cO))\|^2 \, d\cO.
\]
It was shown in the proof of Lemma \ref{lem Sobolev indep f} that this expression equals
\[
C'' \int_{M} \bigl( \|\XoneH(m)\|^2 + 1\bigr) \|fs(m)\|^2 \, dm = C'' \bigl(  (\|\XoneH\|^2+1)fs, fs \bigr)_0,
\]
and the claim follows.
\end{proof}

\noindent
\emph{Proof of Proposition \ref{prop adjoint D}.}
Because of Lemma \ref{lem est L2 clifford}, the (real part of the) operator $fs \mapsto \sqrt{-1}c(df)c(\XoneH)s$ is bounded below by the multiplication operator
\[
-C'' \bigl(\|\XoneH\|^2 + 1 \bigr)
\]
on $f\Gamma^{\infty}_{tc}(E)^G$.
In addition, we had
\[
A \geq -C\bigl(\|\XoneH\|^2 + 1 \bigr).
\]
Lemma \ref{lem Dt*Dt} therefore implies that
\[
tB := (\tilD^{L^p}_t)^* \tilD^{L^p}_t  -\left( (\tilD^{L^p})^* \tilD^{L^p} + 4\pi p t \HH + \frac{t^2}{4} \|\XoneH\|^2 \right)
\geq -t(C+C'') \bigl(\|\XoneH\|^2 + 1 \bigr).
\]
(Note that the operator $B$ is symmetric with respect to the $L^2$-inner product.)
\hfill $\square$


\subsection{Proofs of localization estimates} \label{sec large p}

Let us prove Propositions \ref{prop loc V} and \ref{thm loc global}.

\medskip
\noindent\emph{Proof of Proposition \ref{prop loc V}.}
Let a family of inner products on $\kg^*$ and a set $V \subset M$ as in Proposition \ref{prop choice family} be given. Let the operator $B$ and the constant $C' > 0$ be as in Proposition \ref{prop adjoint D}.
Choose any number $\varepsilon > 0$ and set $t_0 := 8C' + 4\varepsilon$. Then for all $t \geq t_0$ and $m \in M \setminus V$, the fact that $\|\XoneH(m)\| \geq 1$ implies that
\[
\begin{split}
 tB_m + 4\pi t\HH(m) + \frac{t^2}{4} \|\XoneH(m)\|^2 &\geq t \left(  \left( \frac{t}{4} - C'\right)\|\XoneH(m)\|^2 - C' + 4\pi \HH(m)   \right) \\ 
	&\geq t \left(   \frac{t}{4} - 2C' \right) \\
	& \geq \varepsilon t.
\end{split}
\]

So for all such $t$, and all $s \in \Omega^{0,*}_{tc}(M; L)^G$ with $\supp(s) \subset M \setminus V$, 
\[
\begin{split}
\| \tilD_t^{L} (fs) \|^2_0 &= \Bigl(  \bigl( \tilD_t^{L}\bigr)^* \tilD_t^{L} (fs), fs  \Bigr)_0 \\
	&\geq \Bigl( \bigl( \tilD^{L}\bigr)^* \tilD^{L}  (fs), fs  \Bigr)_0 + \varepsilon t  \|fs\|_0^2 \\
	&\geq \|fs\|_1^2 + (\varepsilon t  - 1) \|fs\|_0^2.
\end{split}
\]
If one sets  $\widetilde{C} := \min(\varepsilon, 1)$ and $b := 1$, the latter expression is at least equal to
\[
 \widetilde{C} \bigl( \|fs\|^2_1 + (t-b)\|fs\|^2_0 \bigr).
\]
\hfill $\square$

\medskip


\noindent \emph{Proof of Proposition \ref{thm loc global}.}
Consider a family of inner products on $\kg^*$ as in Proposition \ref{prop choice family}.
Fix a $G$-invariant open neighborhood $U$ of $\mu^{-1}(0)$.
 Since $\HH \geq 1$ outside the set $V$, and $\HH$ is positive and $G$-invariant on the cocompact set $\overline{V} \setminus U$, there is a $\zeta > 0$ such that $\HH \geq \zeta$ outside $U$.

Let the operator $B$ and the constant $C' > 0$ be as in Proposition \ref{prop adjoint D}.
Let $p_0 \in \N$ be such that
\[
\varepsilon := 4\pi \zeta p_0 - C' > 0.
\]
Then for all $p \geq p_0$ and $t \geq t_0 := 4C'$, one has for all $m \in M \setminus U$,
\[
\begin{split}
 tB_m + 4\pi pt\HH(m) + \frac{t^2}{4} \|\XoneH(m)\|^2 &\geq t \left(  \left( \frac{t}{4} - C'\right)\|\XoneH(m)\|^2 - C' + 4\pi \zeta p   \right) \\ 
	&\geq \varepsilon t.
\end{split}
\]
So for all such $p$ and $t$, and all $s \in \Omega^{0,*}_{tc}(M; L^p)^G$ with $\supp(s) \subset M \setminus U$, analogously to the proof of Proposition \ref{prop loc V}, we find that
\[
\begin{split}
\| \tilD_t^{L^p} (fs) \|^2_0 
	&\geq \Bigl(   \bigl( \tilD_t^{L^p}\bigr)^* \tilD_t^{L^p} (fs), fs  \Bigr)_0 + \varepsilon t \|fs\|_0^2 \\
	&= \|fs\|_1^2 + (\varepsilon t -1) \|fs\|_0^2.
\end{split}
\]
As in the proof of Proposition \ref{prop loc V}, the latter expression is at least equal to
\[
\widetilde{C} \bigl( \|fs\|^2_1 + (t-b)\|fs\|^2_0 \bigr),
\]
for  $\widetilde{C} := \min(\varepsilon, 1)$ and $b := 1$.
\hfill $\square$

\begin{remark}
In the proof of Proposition \ref{thm loc global}, it was necessary to take both $p$ and $t$ large enough, so that the term $ \frac{t^2}{4} \|\XoneH(m)\|^2$ compensates for the term $tB_m \geq -tC'(\|\XoneH(m)\|^2 + 1)$. This is different from the arguments in \cite{TZ98} and \cite{MZ} for large powers of $L$, since the deformation vector field and the operator $A$ are bounded in the compact and cocompact situations considered there. Then it is enough that just $p$ is large.
\end{remark}

\subsection{Discrete localized spectrum}\label{sec:localized spectrum}

In \cite{TZ98}, the fact that the restriction of $\bigl( D^{L^p}_t\bigr)^2$ to $G$-invariant sections has discrete spectrum is used. This fact generalizes to the current setting as follows.
\begin{lemma} \label{lem spec discr}
  Let $\lambda > 0$.
Then for $t$ and $p$ large enough, the intersection of the interval $]-\infty, \lambda]$ with the spectrum of 
the restriction of $\bigl( D^{L^p}_t\bigr)^2$ to $G$-invariant sections
 is discrete, and the corresponding eigenspaces are finite-dimensional.
\end{lemma}
\begin{proof}
Let $U$ be a $G$-invariant relatively cocompact open neighborhood of $\mu^{-1}(0)$, on which $G$ acts freely. 
By Proposition \ref{thm loc global},
there are $C>0$, $b>0$, $t_0>0$ and $p_0 > 0$, such that for all $t>t_0$ and $p \geq p_0$, and all $s \in \Omega^{0,*}(M; L)^G$ with support disjoint from $U$,
\begin{equation} \label{eq loc discr spec}
\| \tilD_t^{L^p} (fs) \|^2_0 \geq C\bigl( \|fs\|^2_1 + (t-b)\|fs\|^2_0 \bigr).
\end{equation}
Let $\chi_{]-\infty, \lambda]}$ be the characteristic function of the interval $]-\infty, \lambda]$.
Set
\[
E_t(\lambda) := \image \left(  \chi_{]-\infty, \lambda]} \Bigl( \bigl(\tilD^{L^p}_t\bigr)^2\Bigr) \right).
\]
Then for all 
sections $\sigma \in E_t(\lambda)$,
\[
\| \tilD_t^{L^p} (\sigma) \|_0 \leq \lambda^{1/2} \|\sigma\|_0.
\]
Hence by \eqref{eq loc discr spec}, one has
\[
\lambda\|\sigma\|_0^2 \geq C(t-b)\|\sigma\|_0^2,
\]
if $\supp(\sigma) \subset M\setminus U$. For 
\[
t > t_0(\lambda) := \max \left( t_0,  \frac{\lambda + b}{C}\right),
\]
this implies that $\sigma = 0$. That is, for $t \geq t_0(\lambda)$, sections in $E_t(\lambda)$ localize to $U$. Using Rellich's lemma on the relatively compact set $U/G$, we see that the space of all $G$-invariant sections in $E_t(\lambda)$ is spanned by eigensections, and the claim follows.
\end{proof}



\section{Dirac operators on $M$ and $M_0$} \label{sec Dirac M M0}

In \cite{TZ98}, Tian and Zhang prove a relation between the deformed Dirac operator $D^{L^p}_t$ on $M$, and a Dirac-type operator $D^{L^p_G}_Q$ on $M_0$. This relation allows them to use apply the techniques in \cite{BL91} to the operator $D^{L^p}_t$. We will generalize this relation to the noncompact case.

A version of this relation in the cocompact case was (implicitly) used in \cite{MZ}. In the present setting, one basically arrives at the cocompact situation after localizing to a relatively cocompact set $U$. The authors though it worthwhile to include an explicit discussion for the operator $\tilD^{L^p}_t$, however.


\subsection{Vector bundles on orbit spaces}

We begin by briefly recalling some facts and notation concerning vector bundles on orbit spaces of free actions, induced by equivariant vector bundles on the space acted on.

Let $U$ be a manifold on which a Lie group $G$ acts properly and freely. (We will apply what follows to an open neighborhood $U$ of $\mu^{-1}(0)$ in $M$.) Let $q: U \to U/G$ be the quotient map. Let $E \to U$ be a $G$-vector bundle. As in Subsection \ref{sec free},  let 
\[
E_G \to U/G
\]
be the induced vector bundle, such that $E\cong q^*E_G$ as $G$-vector bundles over $U$. 
Consider the Sobolev spaces $W^k_f(E)^G$ as in Definition \ref{def Hkf}. Let
\[
R = \psi^{-1}: W^k_f(E)^G \xrightarrow{\cong}  W^k(E_G) 
\]
be the inverse of the unitary isomorphism of Lemma \ref{lem iso sobolev}. (We will use the same notation for the restriction of $R$ to the dense subspace $f\Gamma^{\infty}_{tc}(E)^G$.)

For any $G$-invariant submanifold $N \subset U$, consider the inclusion map $i: N \hookrightarrow U$, and the induced inclusion map
\[
i_G: N/G \hookrightarrow U/G.
\]
This induces the restriction map
\[
i_G^*: \Gamma^{\infty}(U/G, E_G) \to \Gamma^{\infty}(N/G, i_G^*E_G).
\]

In the setting of Subsection \ref{sec Dirac}, let now $U$ be a $G$-invariant open neighborhood of $\mu^{-1}(0)$, on which $G$ acts freely. Let $E = \mybigwedge^{0,*}T^*U \otimes (L^p|_U)$, and $N = \mu^{-1}(0)$. Then, as in Subsection 3e of \cite{TZ98}, one has the projection map
\[
\pi: i_G^*  \bigl(\mybigwedge^{0,*}T^*U \otimes (L^p|_U)\bigr)_G \to \mybigwedge^{0,*}T^*M_0 \otimes L_0^p,
\]
defined as follows. Let $N_G \to M_0$ be the normal bundle to $M_0$ in $U/G$. Consider  the almost complex structure $J_G$ on $(TU)_{U/G}|_{M_0}$ induced by the almost complex structure $J$ on $M$. Then 
\[
i_G^* (TU)_G = N_G \oplus J_G N_G \oplus TM_0,
\]
so
\[
i_G^*  \bigl(\mybigwedge^{0,*}T^*U\bigr)_G \cong \mybigwedge^{0,*}\bigl(N_G^* \oplus J_G N_G^*\bigr) \otimes \mybigwedge^{0,*} T^*M_0.
\]
The map $\pi$ is defined via this identification, as projection onto the term
\[
\mybigwedge^{0,0}\bigl(N_G^* \oplus J_G N_G^*\bigr) \otimes \mybigwedge^{0,*} T^*M_0 \otimes i_G^*(L^p|_U)_G \cong \mybigwedge^{0,*} T^*M_0 \otimes L_0^p.
\]
Let 
\[
\iota: \mybigwedge^{0,*}T^*M_0 \otimes L_0^p \hookrightarrow  i_G^*  \bigl(\mybigwedge^{0,*}T^*U \otimes (L^p|_U)\bigr)_G 
\]
be the embedding induced by the same identification, so that $\pi \circ \iota$ is the identity on $\mybigwedge^{0,*}T^*M_0 \otimes L_0^p$.

In the next subsections, we will see how the deformed Dirac operator $D_t^{L^p}$ on $M$ is related to a Dirac-type operator $D^{L_0^p}_Q$ on $M_0$ by the maps
\begin{equation} \label{eq maps M M0}
\begin{split}
f\Omega^{0,*}_{tc}(U; L^p|_U)^G &\xrightarrow{R} \Gamma^{\infty}_c \left(U/G,  \bigl(\mybigwedge^{0,*}T^*U \otimes (L^p|_U)\bigr)_G\right) \\
&\xrightarrow{i_G^*}  \Gamma^{\infty} \left(M_0,  i_G^*\bigl(\mybigwedge^{0,*}T^*U \otimes (L^p|_U)\bigr)_G\right) \\
&\longleftrightarrow^{\hspace{-4.5mm}\pi}_{\hspace{-4.2mm}\iota} \Omega^{0,*}(M_0; L_0^p).
\end{split}
\end{equation}
(Since $M_0$ is compact, all sections of vector bundles over $M_0$ are compactly supported. Therefore,  the subscript $c$ is dropped in the notation for spaces of sections of such bundles.)

\subsection{Intermediate operators}

There are several operators on the vector bundles considered in the previous subsection that are relevant to our purposes.
These operators are related to each other by the maps $R$, $i_G^*$ and $\pi$ in \eqref{eq maps M M0}. Let us define these operators.

Choose a $G$-invariant local orthonormal frame $\{e_1, \ldots, e_{d_M}\}$ for $TU$, such that $\{e_{d_{M/G}+1}, \ldots, e_{d_M}\}$ is a frame for the vertical tangent bundle $\ker(Tq)$.
For a $G$-invariant vector field $v$ on $U$, consider  the vector field $q_*v$ on $U/G$, given by
\[
(q_*v)_{q(m)} := T_mq(v_m).
\]
For $j = 1, \ldots, d_{M/G}$, write $f_j := q_*e_j$. Then $\{f_1, \ldots, f_{d_{M/G}}\}$ is an orthonormal frame for $T(U/G)$. Suppose that $i_G^*f_1, \ldots i_G^*f_{d_{M_0}}$ is an orthonormal frame for $TM_0$. 

Consider the operator $D_G$ on 
\[
\Gamma^{\infty}_c \left(U/G,  \bigl(\mybigwedge^{0,*}T^*U \otimes (L^p|_U)\bigr)_G\right)
\]
and the operator $i_G^*D_G$ on 
\[
\Gamma^{\infty} \left(M_0,  i_G^*\bigl(\mybigwedge^{0,*}T^*U \otimes (L^p|_U)\bigr)_G\right),
\]
given by
\[
\begin{split}
D_G &:= \sum_{j=1}^{d_{M/G}} c(f_j) \nabla^{ \left(\bigwedge^{0,*}T^*U \otimes (L^p|_U)\right)_G}_{f_j}; \\
i_G^*D_G &:= \sum_{j=1}^{d_{M/G}} c(i_G^*f_j) i_G^*\left(\nabla^{ \left(\bigwedge^{0,*}T^*U \otimes (L^p|_U)\right)_G}\right)_{i_G^*f_j } \\
&= \sum_{j=1}^{d_{M_0}} c(i_G^*f_j) \nabla^{ i_G^*\left(\bigwedge^{0,*} T^*U \otimes (L^p|_U)\right)_G}_{i_G^*f_j }. 
\end{split}
\]
Here, for any $G$-vector bundle $E \to U$, with a $G$-invariant connection $\nabla^E$, the connection $\nabla^{E_G}$ on $E_G$ is defined by commutativity of
\[
\xymatrix{
\Gamma^{\infty}(E_G) \ar[d]^-{q^*} \ar[r]^-{\nabla^{E_G}_{q_*v}} & \Gamma^{\infty}(E_G)  \ar[d]^-{q^*} \\
\Gamma^{\infty}(E)^G \ar[r]^-{\nabla^E_v} & \Gamma^{\infty}(E)^G, \\
}
\]
for all $G$-invariant vector fields $v$ on $U$. 

Also, consider the operator
\[
B:= \sum_{j = d_{M/G}+1}^{d_M} c(e_j) \nabla^{\bigwedge^{0,*}T^*U \otimes (L^p|_U)}_{e_j}
\]
on $\Omega^{0,*}(U; L^p)^G$. Because the vector fields $e_{d_{M/G}+1}, \ldots, e_{d_M}$ are tangent to $G$-orbits, $B$ has the following property.
\begin{lemma} \label{lem B endom}
The operator $B$ is given by a vector bundle endomorphism of $\mybigwedge^{0,*}T^*U \otimes (L^p|_U)$.
\end{lemma}
\begin{proof}
See Lemma 3.3 in \cite{TZ98}.
\end{proof}
Let $B_G$ be the operator on 
\[
\Gamma^{\infty}_c \left(U/G,  \bigl(\mybigwedge^{0,*}T^*U \otimes (L^p|_U)\bigr)_G\right)
\]
induced by the $G$-equivariant vector bundle endomorphism $B$, and let the operator $i_G^*B_G$ on
\[
\Gamma^{\infty} \left(M_0,  i_G^*\bigl(\mybigwedge^{0,*}T^*U \otimes (L^p|_U)\bigr)_G\right)
\]
be the restriction of $B_G$ to $M_0$.

\subsection{An operator on $M_0$}

Using the operators $D_G$, $i_G^*D_G$, $B_G$ and $i_G^*B_G$ from the previous subsection, we define a Dirac-type operator $D^{L_0}_Q$ on $M_0$, and show that the maps \eqref{eq maps M M0} relate this operator to $\tilD^L_t$. This is a version of Corollary 3.6 and Definition 3.12 in \cite{TZ98} for the noncompact setting. 

The first step is a relation between the undeformed Dirac operator $\tilD^{L^p}$ and the operator $D^{L^p}_Q$ on $M_0$ defined by
\[
D^{L^p}_Q = \pi \circ \bigl( i_G^*D_G + i_G^*B_G \bigr) \circ \iota: \Omega^{0,*}(M_0; L_0^p) \to \Omega^{0,*}(M_0; L_0^p).
\]
\begin{proposition}
The following diagram commutes:
\begin{equation} \label{eq diag DL DQ}
\xymatrix{
f\Omega^{0,*}_{tc}(U; L^p|_U)^G  \ar[d]_-{R}  \ar[rr]^-{\tilD^{L^p}}
	& &   f\Omega^{0,*}_{tc}(U; L^p|_U)^G  \ar[d]^-{R}\\
 \Gamma^{\infty}_c \left(U/G,  \bigl(\mybigwedge^{0,*}T^*U \otimes (L^p|_U)\bigr)_G\right) \ar[d]_-{i_G^*} \ar[rr]^-{D_G + B_G}&  & 
 	 \Gamma^{\infty}_c \left(U/G,  \bigl(\mybigwedge^{0,*}T^*U \otimes (L^p|_U)\bigr)_G\right) \ar[d]^-{i_G^*} \\
\Gamma^{\infty} \left(M_0,  i_G^*\bigl(\mybigwedge^{0,*}T^*U \otimes (L^p|_U)\bigr)_G\right)\ar[rr]^-{i_G^*D_G + i_G^*B_G} & &  
	\Gamma^{\infty} \left(M_0,  i_G^*\bigl(\mybigwedge^{0,*}T^*U \otimes (L^p|_U)\bigr)_G\right) \ar[d]^-{\pi}\\
\Omega^{0,*}(M_0; L_0^p) \ar[u]_{\iota} \ar[rr]^-{D^{L_0^p}_Q}&  & \Omega^{0,*}(M_0; L_0^p).
}
\end{equation}
\end{proposition}
\begin{proof}
The bottom part of Diagram \eqref{eq diag DL DQ} commutes by definition of $D^{L^p}_Q$.
The middle part of the diagram commutes 
by definition of the pulled-back connection $\nabla^{ i_G^*\left(\bigwedge^{0,*} T^*U \otimes (L^p|_U)\right)_G}$, and the facts that
 $i_G^*B_G$ is the restriction to $M_0$ of $B_G$, and $c(i_G^*f_j)$ is the restriction of $c(f_j)$.

To show that the top part commutes, let $s_G \in  \Gamma^{\infty}_c \left(U/G,  \bigl(\mybigwedge^{0,*}T^*U \otimes (L^p|_U)\bigr)_G\right)$ be given, and set $s:= q^*s_G \in \Omega^{0,*}_{tc}(U; L^p|_U)^G$. Then
\[
\begin{split}
(\tilD^{L^p} \circ R^{-1}) s_G &= fD^{L^p} s \\
	&= f  \sum_{j=1}^{d_{M}} c(e_j) \nabla^{ \left(\bigwedge^{0,*}T^*U \otimes (L^p|_U)\right)_G}_{e_j} s \\
	&= f  \sum_{j=1}^{d_{M/G}} c(e_j) \nabla^{ \left(\bigwedge^{0,*}T^*U \otimes (L^p|_U)\right)_G}_{e_j} s + fBs.
\end{split}
\]
The first of the latter  two terms equals
\[
fq^*\left(D_G  s_G\right) = (R^{-1} \circ D_G) s_G.
\]
The second term equals $fq^*(B_G s_G)$, so that indeed $\tilD^{L^p} \circ R^{-1} = R^{-1} \circ (D_G + B_G)$.
\end{proof}

An important property of the operator $D^{L_0^p}_Q$ is that it has the same index as the $\Spin^c$-Dirac operator $D^{L_0^p}$ on $M_0$ coupled to $L_0^p$.
\begin{lemma}\label{lem:same index}
One has
\[
\ind D^{L_0^p}_Q = \ind D^{L_0^p}.
\]
\end{lemma}
\begin{proof}
By Lemma \ref{lem B endom}, the operator $i_G^*D_G + i_G^*B_G$ in the third horizontal arrow in \eqref{eq diag DL DQ} has the same principal symbol as the term $i_G^*D_G$ on its own. 
The principal symbol of that term induces the  Clifford action on $\mybigwedge^{0,*}T^*M_0\otimes L_0^p$. 
 Hence the principal symbol of $D^{L_0^p}_Q$ is given by the Clifford action, and is equal to the principal symbol of $D^{L_0^p}$. Because $M_0$ is compact, the Fredholm indices of $D^{L_0^p}_Q$ and $D^{L_0^p}$ are therefore equal.
\end{proof}

The final step is to relate the deformed Dirac operator $\tilD^{L^p}_t$ on $M$ to the operator $D^{L_0^p}_Q$ on $M_0$. To this end, note that
\[
R\circ c(\XoneH) = c(q_* \XoneH) \circ R, 
\]
and that 
\[
i_G^* \circ c(q_* \XoneH)  =  c\bigl(i_G^* (q_* \XoneH)\bigr) = 0.
\]
For the last equality, we have used the fact that $\XoneH \equiv 0$ on $\mu^{-1}(0)$, by Lemma \ref{lem X phi}. We therefore obtain the following result.
\begin{corollary} \label{cor DLT DQ}
The following diagram commutes:
\begin{equation} \label{eq diag Dt DQ}
\xymatrix{
f\Omega^{0,*}_{tc}(U; L^p|_U)^G  \ar[d]_-{ i_G^* \circ R}  \ar[rr]^-{\tilD^{L^p}_t}
	& &    f\Omega^{0,*}_{tc}(U; L^p|_U)^G  \ar[d]^-{i_G^* \circ R}\\
	\Gamma^{\infty} \left(M_0,  i_G^*\bigl(\mybigwedge^{0,*}T^*U \otimes (L^p|_U)\bigr)_G\right)\ar[rr]^-{i_G^*D_G + i_G^*B_G} &   &
	\Gamma^{\infty} \left(M_0,  i_G^*\bigl(\mybigwedge^{0,*}T^*U \otimes (L^p|_U)\bigr)_G\right) \ar[d]^{\pi}\\
\Omega^{0,*}(M_0; L_0^p) \ar[rr]^-{D^{L_0^p}_Q}  \ar[u]_{\iota}& & \Omega^{0,*}(M_0; L_0^p).
}
\end{equation}
\end{corollary}

%

To generalize the methods of \cite{TZ98} to the present setting, the final ingredients one needs are versions of  Remarks 3.7 and 3.8 in \cite{TZ98}. These remarks generalize to the current setting because the form \eqref{eq X1H frame} of the vector field $\XoneH$ is analogous to (1.19) in \cite{TZ98}.
Indeed, Remark 3.7 in \cite{TZ98} generalizes because the vector field $\XoneH$ is still tangent to orbits. Hence, for any local frame $\{f_1, \ldots, f_{d_{M/G}}\}$ of $T(U/G)$, the operators $c(q_*\XoneH)$ and $c(f_j)$ anticommute.  Remark 3.8 in \cite{TZ98} can be generalized because of the following estimate.
\begin{lemma}
There is a neighborhood $U$ of $\mu^{-1}(0)$ and a constant $C>0$ such that 
\[
\HH|_U \leq C\|\XoneH|_U\|^2.
\]
\end{lemma}
\begin{proof}
Since $G$ acts freely on a neighborhood of $\mu^{-1}(0)$, the vector fields $V_j$ are linearly independent there.
Hence every point $m \in \mu^{-1}(0)$ has a neighborhood $U_m$ that admits a vector bundle automorphism $B$ of $TU_m$ such that $\{BV_1|_{U_m}, \ldots, BV_{d_G}|_{U_m}\}$ is orthonormal. Then, on $U_m$, \eqref{eq def muj} and \eqref{eq X1H frame} imply that
\[
\HH =\Sj \mu_j^2 =  \Bigl\| \Sj \mu_j BV_j \Bigr\|^2 \leq \|B\|^2 \|\XoneH\|^2 \leq C_m \|\XoneH\|^2,
\]
if $U_m$ is chosen small enough so that there is a $C_m > 0$ such that $\|B\| \leq C_m$ on $U_m$. 

Since $\HH$ and $\|\XoneH\|^2$ are $G$-invariant, we obtain the estimate $\HH \leq C_m \|\XoneH\|^2$ on $G \cdot U_m$, for any $m \in \mu^{-1}(0)$. Since $\mu^{-1}(0)$ is cocompact, one can cover it with finitely many such sets $G\cdot U_m$.
\end{proof}
Because of this lemma, one has
\[
c(q_*\XoneH)^2 \geq \frac{1}{C} \HH_G
\]
on $U/G$, where $\HH_G$ is the function on $U/G$ induced by $\HH$. This generalizes Remark 3.8 in \cite{TZ98}, which is a version of Proposition 8.14 in \cite{BL91}.

Because of  Proposition \ref{thm loc global}, Lemma \ref{lem spec discr} and Corollary \ref{cor DLT DQ}, 
and the generalizations of Remarks 3.7 and 3.8 in \cite{TZ98} mentioned above,
the methods of \cite{TZ98} generalize to the noncompact setting considered here.
The key result in \cite{TZ98} is Theorem 3.13, which is an analogue of (9.156) in \cite{BL91} in the $\Spin^c$ setting. A generalization of that result to the cocompact setting was used (implicitly) at the end of Section 3 in \cite{MZ}. Since, in the present setting, $M_0$ is compact, Proposition \ref{thm loc global} implies that the kernel of the Dirac operator $\tilD^{L^p}_t$ localizes to a relatively cocompact neighborhood of $\mu^{-1}(0)$. There, one applies the same generalization of Theorem 3.13 in \cite{TZ98} as the one used in \cite{MZ}.

More precisely, using the notation in Section 
\ref{sec:localized spectrum}, recall from Lemma \ref{lem spec discr} that given $\lambda>0$, there is 
$t_0=t_0(\lambda)>0$
such that for all $t\ge t_0(\lambda)$, the $G$-invariant part of the space $E_t(\lambda)$ is spanned by eigensections of $(D^{L^p}_t)^2$. 
Then one has the following result.
\begin{theorem}\label{main:BL}
Let $\lambda>0$ be such that there are no eigenvalues of $(D^{L_0^p}_Q)^2$ in the interval $(0, \lambda]$.
Then there is $t_0=t_0(\lambda)>0$ such that for all $t \ge t_0(\lambda)$, one has
\begin{equation}
\dim E_t(\lambda) = \dim \ker(D^{L_0^p}_Q).
\end{equation}
\end{theorem}
Since the positive and negative eigenvalues of $D^{L^p}_t$ of absolute value at most $\lambda$ are in bijection with each other, one has, for $t \geq t_0(\lambda)$,
\begin{align*}
\ind_G(D^{L^p}_t) & = \ind(D^{L_0^p}_Q) \qquad \text{by Theorem \ref{main:BL},}\\
& =  \ind(D^{L_0^p}) \qquad \text{by Lemma \ref{lem:same index}.}
\end{align*}
Thus one obtains a proof of Theorem \ref{thm [Q,R]=0}.

\appendix
%


\section{Elliptic regularity and transversally $L^2$-kernels} \label{app reg L2t}

Consider the setting of Proposition \ref{prop Fredholm}. Recall that the transversally $L^2$-kernels of the operators $D_{\pm}$ were defined as 
\[
\ker_{L^2_T}(D_{\pm}) = \{s \in \Gamma^{\infty}(E_{\pm}) \cap L^2_T(E); Ds = 0\}.
\]
The space $L^2_T(E)$ of transversally $L^2$-sections of $E$ was defined in Definition \ref{def transv L2}.
One can give an explicit characterization of the $G$-invariant index of the operator $D$ in terms of its transversally $L^2$-kernel as follows.
\begin{proposition} \label{prop L2t index app}
In the situation of Proposition \ref{prop Fredholm}, one has
\[
\ind_G(D) = \dim \bigl( \ker_{L^2_T}(D_+) \bigr)^G -  \dim \bigl( \ker_{L^2_T}(D_-) \bigr)^G.
\]
\end{proposition}

The proof of this proposition is based on a version of elliptic regularity. It will be convenient to use 
a slightly different realization of the
Sobolev spaces $W^k_f(E)^G$ from the one used in Section \ref{sec index}. This realization is defined in terms of the Sobolev norms $\|\cdot\|^f_k$ on $\Gamma^{\infty}_{tc}(E)$, 
equal to
\[
\bigl(\|s\|^k_f \bigr)^2 := \sum_{j=0}^k \|fD^js\|_{L^2(E)}^2 = \|fs\|_{W^k_f(E)^G}^2,
\]
for $s \in \Gamma^{\infty}_{tc}(E)^G$.
Let $\widetilde{W}^k_f(E)^G$ be the completion of $\Gamma^{\infty}_{tc}(E)^G$ in the norm $\|\cdot\|^f_k$. Multiplying by $f$ then extends to a unitary isomorphism $\widetilde{W}^k_f(E)^G \cong W^k_f(E)^G$. The operator $D$ on $\Gamma^{\infty}_{tc}(E)^G$ extends continuously to an operator
\[
D_f: \widetilde{W}^1_f(E)^G \to \widetilde{W}^0_f(E)^G.
\]
The isomorphism $\widetilde{W}^k_f(E)^G \cong W^k_f(E)^G$ just mentioned intertwines this operator and the operator $\tilD$. 

Let $\varphi \in C^{\infty}_c(M)$ be a function whose compact support is contained in the interior of $\supp(f)$. Then
\[
\varepsilon := \min_{m \in \supp(\varphi)} |f(m)| > 0.
\]
\begin{lemma} \label{lem mphi bdd}
Consider the multiplication operator by $\varphi$,
\[
m_{\varphi}: \Gamma^{\infty}_{tc}(E)^G \to \Gamma^{\infty}_c(E).
\] 
It is bounded with respect to the Sobolev norm $\|\cdot\|^f_k$ defined above on the domain, and the norm $\|\cdot\|_{W^k(E)}$ on the codomain defined by 
\begin{equation} \label{eq sob norm E}
\|s\|^2_{W^k(E)} := \sum_{j=0}^k \|D^js\|^2_{L^2(E)},
\end{equation}
for $s \in \Gamma^{\infty}_c(E)$.
\end{lemma}
\begin{proof}
For $k = 0$, we note that for all $s \in \Gamma^{\infty}_{tc}(E)^G$,
\[
\|\varphi s\|_{L^2(E)} \leq \frac{1}{\varepsilon} \|\varphi f s\|_{L^2(E)} \leq \frac{\|\varphi\|_{\infty}}{\varepsilon} \|s\|^f_0.
\] 
For $k = 1$, one has
\[
\begin{split}
\|D \varphi s\|_{L^2(E)} &= \|\varphi Ds + \sigma_D(d\varphi)s\|_{L^2(E)} \\
&\leq 	\frac{1}{\varepsilon} \left( \|\varphi f Ds\|_{L^2(E)} + \| \sigma_D(d\varphi) fs\|_{L^2(E)} \right) \\
	&\leq \frac{1}{\varepsilon} \left( \|\varphi \|_{\infty}     \|s\|_1^f + \| \sigma_D(d\varphi) \| \cdot  \|s\|^f_0 \right).
\end{split}
\]
Hence $m_{\varphi}$ is also bounded with respect to the first Sobolev norms.

For general $k$, one similarly notes that $D^k \varphi s$ equals a sum of the form
\[
\sum_{j=0}^k R_j D^j s,
\]
where $R_j$ is a repeated commutator of $D$ and $m_{\varphi}$. All these commutators are bounded, since $\varphi$ is compactly supported. Hence
\[
\|D^k \varphi s\|_{L^2(E)} \leq \frac{1}{\varepsilon}\sum_{j=0}^k \|fR_j D^js\| 
	\leq  \frac{1}{\varepsilon} \sum_{j=0}^k \|R_j\| \cdot  \|s\|^f_j.
\]
We conclude that $m_{\varphi}$ is bounded with respect to the $k$'th Sobolev norms used, for all $k$. 
\end{proof}

Let $W^k(E)$ be the completion of $\Gamma^{\infty}_c(E)$ in the norm given by \eqref{eq sob norm E}.
Because of Lemma \ref{lem mphi bdd}, we obtain the continuous extension
\[
m_{\varphi}^{(k)}: \widetilde{W}^k_f(E)^G \to W^k(E).
\]
The diagram
\begin{equation} \label{eq diag mphi incl}
\xymatrix{
\widetilde{W}^k_f(E)^G \ar@{^{(}->}[r] \ar[d]^{m^{(k)}_{\varphi}}& \widetilde{W}^{k-1}_f(E)^G\ar[d]^{m^{(k-1)}_{\varphi}} \\
W^k_f(E) \ar@{^{(}->}[r] & W^{k-1}(E)
}
\end{equation}
commutes on the dense subspace $\Gamma^{\infty}_{tc}(E)^G$ of $\widetilde{W}^k_f(E)^G$. Because all maps in the diagram are bounded, it therefore commutes on all of $\widetilde{W}^k_f(E)^G$.

\begin{lemma}
The kernel of the operator $D_f$ consists of smooth, $G$-invariant sections of $E$.
\end{lemma}
\begin{proof}
Since $G$ acts trivially on the space $\Gamma^{\infty}_{tc}(E)^G$, the continuous extension of the action to $\widetilde{W}^1_f(E)^G$ is trivial as well. Hence the kernel of $D_f$ consists of  $G$-invariant sections of $E$. 

Let $s \in \ker(D_f)$ be given. 
Since $D_f^j s = 0$ for all $j$, the section $s$ is in all Sobolev spaces $\widetilde{W}^k_f(E)^G$. Hence, by commutativity of \eqref{eq diag mphi incl}, we find that
\[
m^{(0)}_{\varphi}(s) = m^{(k)}_{\varphi}(s) \in W^k(E)
\]
for all $k$. Therefore, 
the section $m^{(0)}_{\varphi}(s)$ is smooth. Since $m^{(0)}_{\varphi} s$ is the pointwise product of $\varphi$ and $s$, we conclude that $s$ is smooth on the interior of the support of $\varphi$. Since this argument holds for any function with the properties of $\varphi$, we find that $s$ is smooth on the interior of the support of $f$. By $G$-invariance of $s$, it is smooth everywhere.
\end{proof}

\noindent \emph{Proof of Proposition \ref{prop L2t index app}.}
We have seen that
\[
\ker(\tilD_{\pm}) \cong \ker\bigl( (D_f)_{\pm} \bigr) \subset \Gamma^{\infty}(E)^G.
\]
On smooth sections, the operator $D_f$ is equal to $D$.

Any section $s \in \widetilde{W}^1_f(E)^G$ satisfies
\[
\|fs\|_{L^2(E)} = \|s\|^f_0 \leq \| s \|^f_1
\]
(with equality if and only if $s \in \ker(D_f)$). Since by Lemma \ref{lem Sobolev indep f}, for $G$-invariant $s$, the norm $\|fs\|_{L^2(E)}$ is independent of the cutoff function $f$, one has
\[
\widetilde{W}^1_f(E)^G \subset L^2_T(E).
\]
We conclude that
\[
\ker\bigl( (D_f)_{\pm} \bigr) \subset  \Gamma^{\infty}(E)^G \cap L^2_T(E),
\]
and the claim follows.
\hfill $\square$

\section{Computing the square of $D_t^{L^p}$} \label{app Bochner}

This appendix contains a proof of an explicit expression for the covariant derivative $\nabla_{\XoneH}$. This is the main computation in the proof of Theorem \ref{thm Bochner}. We will use the notation of Section \ref{sec Bochner}. 
\begin{proposition} \label{prop nabla X1H}
For all $s \in \Omega^{0,*}(M; L^p)$ and all $m \in M$, one has
\begin{multline}
(\nabla_{\XoneH} s)(m) = 
	2 \Sj  \mu_j(m) (L_{h_j^*(m)}s)(m) + 4\pi \sqrt{-1} p \HH(m) s(m) \\
	+ \left( \frac{1}{4} \Sk c(e_k) c(\nabla_{e_k} \XoneH)
	+ \frac{1}{2}\Sj \bigl( -c(JV_j)c(V_j) + \sqrt{-1} \|V_j\|^2 \bigr)  \right. \\
	\left. 
	 + \frac{1}{2} \tr \left( \nabla^{T^{1,0}M}\XoneH|_{T^{1,0}M}   \right) \right)(m) s(m) \\
+\sqrt{-1}(A_2 + A_3)(m) s(m),		
\end{multline}
with $A_2$ and $A_3$ as in \eqref{eq def A2} and \eqref{eq def A3}.
\end{proposition}
The proof of this proposition is based on a number of intermediate lemmas.

\begin{lemma} \label{lem 3rd term}
For all $s \in \Omega^{0,*}(M; L^p)$ and $m \in M$, one has
\begin{multline*}
	\frac{1}{2} \Sj \Sk \mu_j(m)  \Bigl(c(e_k) c(\nabla_{e_k} h^*_j(m)^M) \Bigr)(m)s(m)   =  \\
\left(\frac{1}{4} \Sk c(e_k) c(\nabla_{e_k} \XoneH) -\frac{1}{2}\Sj c(\grad \mu_j)c(V_j) \right. \\
	\left. + \frac{1}{2} \Sjk \mu_j c(e_k)c \bigl(  \bigl[e_k, (h^*_j)^M - V_j\bigr] \bigr)	 \right)(m) s(m).
\end{multline*}
\end{lemma}
\begin{proof}
Since the Levi-Civita connection $\nabla$ is torsion-free, we have, for all $j$ and $k$, and all $m \in M$,
\begin{multline} \label{eq 3rd term 0}
 \left(\nabla_{e_k}\bigr(h_j^*(m)\bigr)^M - \nabla_{e_k}V_j \right)(m)
	=  \\ \left(\bigl(\nabla_{h^*_j(m)^M} - \nabla_{V_j}\bigr)e_k \right)(m)+ \bigl[ e_k, h_j^*(m)^M - V_j \bigr] (m)
	= \\ \bigl[ e_k, (h_j^*)^M - V_j \bigr](m),
\end{multline}
since $h_j^*(m)^M_m = V_j(m)$.
So
\begin{multline} \label{eq 3rd term 1}
\frac{1}{2} \Sj \Sk \mu_j(m)  \Bigl(c(e_k) c(\nabla_{e_k} h^*_j(m)^M) \Bigr)(m)s(m) = \\
  \frac{1}{2} \Sj \Sk \mu_j(m)  \Bigl(c(e_k) c(\nabla_{e_k} V_j ) \Bigr)(m)s(m) \\
  +   \frac{1}{2} \Sj \Sk \mu_j(m)  \Bigl(c(e_k) c\bigl( \bigl[ e_k, (h_j^*)^M - V_j \bigr]   \bigr) \Bigr)(m)s(m).
\end{multline}

By \eqref{eq X1H frame}, we find that
\begin{equation} \label{eq 3rd term 2}
\Sj \mu_j c(\nabla_{e_k}V_j) = \frac{1}{2}c(\nabla_{e_k}\XoneH) - \Sj e_k(\mu_j) c(V_j).
\end{equation}
Therefore, the first term on the right hand side of \eqref{eq 3rd term 1} equals
\[
 \frac{1}{4}  \Sk  \Bigl(c(e_k) c(\nabla_{e_k} \XoneH ) \Bigr) (m)s(m) - 
 \frac{1}{2} \Sj \Sk  \Bigl(e_k(\mu_j) c(e_k) c(V_j ) \Bigr)(m)s(m).
\]
Because
\[
\Sk e_k(\mu_j) e_k = \grad \mu_j,
\]
the desired equality follows.
\end{proof}

\begin{lemma} \label{lem grad muj}
One has
\[
\grad \mu_j = JV_j + \langle \mu, Th_j^* \rangle^*. 
\]
\end{lemma}
\begin{proof}
Note that, for all $m \in M$, 
\[
\mu_j(m) = \langle \mu(m), h_j^*(m) \rangle = \mu_{h_j^*(m)}(m).
\]
Hence, by \eqref{eq def mom},
\[
\begin{split}
d_m\mu_j &= d_m\bigl(m' \mapsto \mu_{h_j^*(m)}(m') \bigr) + d_m\bigl(m' \mapsto \mu_{h_j^*(m')}(m) \bigr) \\
	&= \omega_m\bigl(   V_j(m), \relbar \bigr) + \langle \mu, Th_j^*\rangle_m \\
	&= \omega_m\bigl(   V_j(m), \relbar \bigr) + \omega_m\bigl( \relbar  , J  \langle \mu, Th_j^*\rangle_m^* \bigr) \\
	&= \omega_m\bigl( V_j(m) - J\langle \mu, Th_j^*\rangle_m^*, \relbar \bigr).
\end{split}
\]

Now also
\[
\begin{split}
d\mu_j 
	&= \omega(\relbar, J\grad \mu_j) \\
	&= -\omega(J\grad \mu_j, \relbar),
\end{split}
\]
and the claim follows.
\end{proof}

\begin{lemma} \label{lem 4th term}
One has, for all $m \in M$,
\begin{multline*}
  \Sj \mu_j(m) \tr \left( \nabla^{T^{1,0}M} h^*_j(m)^M|_{T^{1,0}M} \right)(m)= \\
\left(\frac{1}{2} \tr \left( \nabla^{T^{1,0}M}\XoneH|_{T^{1,0}M}  \right) - \Sjk e_k^{1,0}(\mu_j) \bigl( V_j, e_k^{1,0} \bigr) \right. \\
 	\left. + \Sjk \mu_j \bigl(  \bigl[e_k^{1,0}, (h^*_j)^M - V_j\bigr], e_k^{1,0}\bigr)\right)(m).
\end{multline*}
\end{lemma}
\begin{proof}
Note that, for all $m \in M$,
\begin{equation} \label{eq 4th term 1}
\tr \left( \nabla^{T^{1,0}M} h^*_j(m)^M|_{T^{1,0}M} \right) = \Sk \left( \nabla^{T^{1,0}M}_{e_k^{1,0}} h^*_j(m)^M, e_k^{1,0}\right).
\end{equation}
Analogously to \eqref{eq 3rd term 0}, we find
\[
 \left(\nabla^{T^{1,0}M}_{e_k^{1,0}}\bigr(h_j^*(m)\bigr)^M \right)(m) =
   \left(\nabla^{T^{1,0}M}_{e_k^{1,0}}V_j \right)(m) + 
 \bigl[ e_k^{1,0}, h_j^*(m)^M - V_j \bigr](m).
\]
Hence, at $m$, \eqref{eq 4th term 1} equals
\[
\tr \left( \nabla^{T^{1,0}M} V_j|_{T^{1,0}M} \right)(m) + 
\Sk \left(  \bigl[ e_k^{1,0}, h_j^*(m)^M - V_j \bigr]  , e_k^{1,0}\right)(m).
\]

As in  \eqref{eq 3rd term 2}, one has
\[
\Sj \mu_j \nabla^{T^{1,0}M}_{e_k^{1,0}}V_j = 
	\frac{1}{2}  \nabla^{T^{1,0}M}_{e_k^{1,0}}\XoneH - \Sj e_k^{1,0}(\mu_j) V_j.
\]
The claim follows.
\end{proof}

\begin{lemma} \label{lem norm Vj}
One has
\[
\Sk (JV_j, e_k^{1,0}) (V_j, e_k^{1,0}) = \frac{1}{2\sqrt{-1}} \|V_j\|^2.
\]
\end{lemma}
\begin{proof}
Since $\frac{1}{2}\left(1+ \frac{J}{\sqrt{-1}} \right)$ is the projection $TM \otimes \C \to T^{1,0}M$, we have
\begin{multline*}
\Sk (JV_j, e_k^{1,0}) (V_j, e_k^{1,0}) = \Sk \left(   \frac{1}{2}\left(1+ \frac{J}{\sqrt{-1}} \right)  JV_j, e_k\right) 
						\left(\frac{1}{2}\left(1+ \frac{J}{\sqrt{-1}} \right)  V_j, e_k \right) \\
	= \frac{1}{4} \left( \left(1+ \frac{J}{\sqrt{-1}} \right)  JV_j,  \left(1+ \frac{J}{\sqrt{-1}} \right)  V_j \right) \\
	= \frac{1}{4} \left( (  JV_j,   V_j) + \left( JV_j, \frac{JV_j}{\sqrt{-1}} \right) 
		+ \left( \frac{-1}{\sqrt{-1}}V_j, V_j\right) + \left(  \frac{-1}{\sqrt{-1}}V_j, \frac{J}{\sqrt{-1}}V_j\right) \right).
\end{multline*}
Since $JV_j$ is orthogonal to $V_j$, the first and last terms in the latter expression vanish. Since the extension of the Riemannian metric to $TM \otimes \C$  was (tacitly) taken to be complex-antilinear in the first coordinate, the remaining two terms add up to $\frac{1}{2\sqrt{-1}} \|V_j\|^2$.
\end{proof}

\noindent \emph{Proof of Proposition \ref{prop nabla X1H}.}
Let $s \in \Omega^{0,*}(M; L^p)$ and all $m \in M$. We saw in \eqref{eq Bochner 1} that
\begin{multline} \label{eq Bochner 1 app}
(\nabla_{\XoneH} s)(m) = 
	2 \Sj  \mu_j(m) (L_{h_j^*(m)}s)(m) + 4\pi \sqrt{-1} p \HH(m) s(m) \\ 
	+ \frac{1}{2} \Sj \Sk \mu_j(m)  \Bigl(c(e_k) c(\nabla_{e_k} h^*_j(m)^M) \Bigr)(m)s(m) \\
 + \Sj \mu_j(m) \tr \left( \nabla^{T^{1,0}M} h^*_j(m)^M|_{T^{1,0}M} \right)(m) s(m).
\end{multline}
By Lemma \ref{lem 3rd term}, the third term on the right hand side of \eqref{eq Bochner 1 app} equals
\begin{multline} \label{eq Bochner 2}
\left(\frac{1}{4} \Sk c(e_k) c(\nabla_{e_k} \XoneH) -\frac{1}{2}\Sj c(\grad \mu_j)c(V_j) \right. \\
	\left. + \frac{1}{2} \Sjk \mu_j c(e_k)c \bigl(  \bigl[e_k, (h^*_j)^M - V_j\bigr] \bigr)	 \right)(m) s(m).
\end{multline}
The gradient of $\mu_j$ was computed in Lemma \ref{lem grad muj}:
\begin{equation} \label{eq grad muj}
\grad \mu_j = JV_j + \langle \mu, Th_j^* \rangle^*. 
\end{equation}
Therefore,  \eqref{eq Bochner 2} equals
\begin{multline} \label{eq Bochner 3}
\left(\frac{1}{4} \Sk c(e_k) c(\nabla_{e_k} \XoneH) -\frac{1}{2}\Sj c(JV_j)c(V_j)  -\frac{1}{2}\Sj c(\langle \mu, Th_j^* \rangle^*)c(V_j) \right. \\
	\left. + \frac{1}{2} \Sjk \mu_j c(e_k)c \bigl(  \bigl[e_k, (h^*_j)^M - V_j\bigr] \bigr)	 \right)(m) s(m).
\end{multline}

By  Lemma \ref{lem 4th term}, the last term on the right hand side of \eqref{eq Bochner 1} equals
\begin{multline} \label{eq Bochner 4}
\left(\frac{1}{2} \tr \left( \nabla^{T^{1,0}M}\XoneH|_{T^{1,0}M}  \right) - \Sjk e_k^{1,0}(\mu_j) \bigl( V_j, e_k^{1,0} \bigr) \right. \\
\left. 	+ \Sjk \mu_j \bigl(  \bigl[e_k^{1,0}, (h^*_j)^M - V_j\bigr], e_k^{1,0}\bigr) \right)(m) s(m).
\end{multline}
Because of \eqref{eq grad muj}, one has
\[
e_k^{1,0}(\mu_j)  = (JV_j, e_k^{1,0}) + \bigl\langle    \langle \mu, Th_j^* \rangle, e_k^{1,0} \bigr\rangle.
\]
It was computed in Lemma \ref{lem norm Vj} that
\[
\Sk (JV_j, e_k^{1,0}) (V_j, e_k^{1,0}) = \frac{1}{2\sqrt{-1}} \|V_j\|^2.
\]
Hence \eqref{eq Bochner 4} equals
\begin{multline} \label{eq Bochner 5}
\left(
\frac{1}{2} \tr \left( \nabla^{T^{1,0}M}\XoneH|_{T^{1,0}M}  \right) 
- \frac{1}{2\sqrt{-1}}\Sj  \|V_j\|^2
-\Sjk    \bigl\langle    \langle \mu, Th_j^* \rangle, e_k^{1,0} \bigr\rangle  \bigl(V_j, e_k^{1,0} \bigr)
 \right. \\
\left. 	+ \Sjk \mu_j \bigl(  \bigl[e_k^{1,0}, (h^*_j)^M - V_j\bigr], e_k^{1,0}\bigr) \right)(m) s(m).
\end{multline}

Using expression \eqref{eq Bochner 3} for the third term in \eqref{eq Bochner 1 app} and expression \eqref{eq Bochner 5} for the last term yields the desired equality.
\hfill $\square$


\section{Estimates used to bound the operator $A$} \label{app bound A}

The facts in this appendix were used in the proof of Proposition \ref{prop choice family}.


\begin{lemma} \label{lem cover}
Let $M$ be a Riemannian manifold. There is an open cover $\{\widetilde{U}_l\}_l$ of $M$ such that 
\begin{itemize}
\item every open set $\widetilde{U}_l$ admits a local orthonormal frame for $TM$;
\item every compact subset of $M$ intersects finitely many of the sets $\widetilde{U}_l$ nontrivially;
\item there is a relatively compact subset $U_l \subset \widetilde{U}_l$ for all $l$, such that $\overline{U}_l \subset \widetilde{U}_l$, and $\bigcup_l U_l = M$.
\end{itemize}
\end{lemma}
\begin{proof}
Let $\{K_j\}_{j=1}^{\infty}$ be a sequence of compact subsets of $M$ such that
\begin{itemize}
\item for all $j$, the set $K_j$ is contained in the interior of $K_{j+1}$;
\item the sets $K_j$ cover $M$.
\end{itemize}
Write $K_{-1} = K_0 = \emptyset$. For every $j \in \N$, the complement of the interior of $K_{j-1}$ in $K_j$ is compact, and hence  admits a cover $\{\widetilde{U}_{j, 1}, \ldots, \widetilde{U}_{j, n_j}\}$ by open subsets of $K_{j+1} \setminus K_{j-2}$ that admit local orthonormal frames of $TM$ and have relatively compact subsets $U_{j, l} \subset \widetilde{U}_{j, k}$ such that $\overline{U}_{j, k} \subset \widetilde{U}_{j, k}$,  which cover $K_j \setminus K_{j-1}$. The cover
\[
\{\widetilde{U}_{j, k}; j \in \N, k \in \{1, \ldots, n_j\}\}
\]
has the desired properties.
\end{proof}

\begin{lemma} \label{lem F}
Let $X$ be a topological space, and let  $\{\widetilde{U}_l\}_l$ be an cover  of $X$ such that 
\begin{itemize}
\item every compact subset of $X$ intersects finitely many of the sets $\widetilde{U}_l$ nontrivially;
\item there is a relatively compact subset $U_l \subset \widetilde{U}_l$ for all $l$, such that $\overline{U}_l \subset \widetilde{U}_l$, and $\bigcup_l U_l = X$.
\end{itemize}
For all $l$, let $h_l$ be a continuous function on $\widetilde{U}_l$.

Let $G$ be a group acting on $X$.
Let $W \subset X$ be a subset whose intersection with every nonempty cocompact subset of $X$ is nonempty and compact. Then there is a positive, $G$-invariant, continuous function $F \in C(X)^G$ such that for all $x \in W$, and for all $l$ such that $x \in U_l$,
\[
h_l(x) \leq F(x).
\]
\end{lemma}
\begin{proof}
For every $l$, the restriction $h_l|_{\overline{U_l}}$ is bounded, so there is a constant $C_l > 0$ such that $h_l \leq C_l$ on $U_l$. For $x \in X$, set
\[
B(x) := \max\{C_l; (G\cdot x) \cap W \cap U_l \not=\emptyset \}
\]
Since the set $(G\cdot x) \cap W$ is compact for all $x$, it intersects finitely many of the sets $U_l$. Hence $B(x)$ is finite for all $x$. In fact, any cocompact set $Z \subset X$ has compact intersection with $W$, and hence intersects finitely many sets $U_l$. Therefore, the function $B$ is bounded on $Z$. In addition, it is $G$-invariant, but not continuous in general. 

Let $B_G$ be the function on $X/G$ induced by $B$. Since it is bounded on compact subsets, there is a continuous function $F_G \in C(X/G)$ such that $B_G \leq F_G$. The pullback $F$ of $F_G$ along the quotient map $X \to X/G$ has the desired property. (It is not necessary that $F$ is positive in general, but one may choose $F$ in this way.)
\end{proof}


\begin{lemma} \label{lem bdd der app}
Let $G$ be a Lie group, acting properly and isometrically on a Riemannian manifold $M$.  Let $\varphi_0, \varphi_1 \in C(M)^G$ be continuous, positive, $G$-invariant functions on $M$. Then there exists a positive, $G$-invariant, smooth function $\psi$ on  $ M$, such that
\[
\begin{split}
\psi &\leq \varphi_0;\\
\|d\psi\| &\leq \varphi_1.
\end{split}
\]
\end{lemma}
\begin{proof}
Let $\{U_j\}_{j=1}^{\infty}$ be a countable, locally finite  open cover of $M$, by $G$-invariant, relatively cocompact open subsets. For example, one can use
\[
U_j := \{m \in M; d(G\cdot m_0, m) \in ]j-2, j[\},
\]
for a fixed orbit $G\cdot m_0 \in M/G$, where $d$ denotes the Riemannian distance on $M$. Since the action is proper, there is a $G$-invariant partition of unity $\{\tilde \chi_j\}_{j=1}^{\infty}$ on $M$, subordinate to the cover $\{U_j\}_{j=1}^{\infty}$ (see e.g.\ \cite{GGK}, Corollary B.33). Define the functions $\chi_j$ by
\[
\chi_j = \frac{\tilde \chi_j}{\max_{m \in \overline{U_j}}  \|d_m \tilde\chi_j\| + 1}.
\]
Then for all $m \in M$ and $j \in \N$,
\[
\begin{split}
 \chi_j(m) &\leq 1;\\
\|d_m\chi_j\| &\leq 1;\\
\sum_{k=1}^{\infty} \chi_k(m) &> 0.
\end{split}
\]

For every $j \in \N$, set
\[
\begin{split}
J_j &:= \{k \in \N; U_k \cap U_j \not=\emptyset\}; \\
a_j &:= \min_{m \in \overline{U_j}} \varphi_0(m); \\
b_j &:= \min_{m \in \overline{U_j}} \varphi_1(m); \\
\alpha_j &:= \min_{k \in J_j}  \frac{\min(a_k, b_k)}{\# J_k}.
\end{split}
\]
Because $\varphi_0$ and $\varphi_1$ are continuous, positive and $G$-invariant, and $\overline{U_j}/G$ is compact, the numbers  $a_j$ and $b_j$, and hence $\alpha_j$, are positive for all $j$.
 Define the function $\psi$ by
\[
\psi:= \sum_{j=1}^{\infty} \alpha_j \chi_j.
\]
This function is $G$-invariant and smooth, and positive everywhere. Let us show that $\psi$ and its derivative satisfy the desired estimates.

Fix  $j \in \N$, and let $m \in U_j$. Then
\[
\psi(m) = \sum_{k \in J_j} \alpha_k \chi_k(m) \leq \sum_{k \in J_j}  \min_{l \in J_k}  \frac{a_l}{\# J_l}.
\]
Since $k \in J_j$ if and only if $j \in J_k$, the summands in the latter sum are at most equal to $\frac{a_j}{\# J_j}$. Therefore, 
\[
\psi(m) \leq \sum_{k \in J_j} \frac{a_j}{\# J_j} = a_j \leq \varphi_0(m).
\]
Similarly, one estimates
\[
\|d_m\psi\| = \Bigl\| \sum_{k \in J_j} \alpha_k d_m \chi_k \Bigr\| 
	\leq \sum_{k \in J_j} \alpha_k 
	\leq \sum_{k \in J_j}  \min_{l \in J_k}  \frac{b_l}{\# J_l} 
	\leq b_j 
	\leq \varphi_1(m).
\]
\end{proof}

\begin{lemma} \label{lem nabla X1H}
Let $v$ be one of the vector fields $e_k$ or $e_k^{1,0}$ in the proof of Proposition \ref{prop choice family}.
Then one has 
\[
\|\nabla_{v} X_1^{\HH_{\psi}} \| \leq 2N_{\psi},
\]
on $(M \setminus V) \cap W$.
\end{lemma}
\begin{proof}
Note that, outside $\Crit_1(\HH)$, one has for all vector fields $v$ and all $k$,
\begin{equation} \label{eq est nabla X1H}
\frac{\|\nabla_{v} X_1^{\HH_{\psi}}\|}{N_{\psi}} \leq \frac{|v(\psi)|}{\psi^2}\frac{1}{N^{1/2}} + \frac{1}{\psi} \frac{\|\nabla_{v} \XoneH \|}{N}.
\end{equation}
Now suppose $v$ is one of the vector fields $e_k$ or $e_k^{1,0}$. Then $v$ has norm at most $1$, so that the first term on the right hand side of \eqref{eq est nabla X1H} can be estimated on $M \setminus V$ by
\[
\frac{|v(\psi)|}{\psi^2}\frac{1}{N^{1/2}}= \frac{|v(\psi^{-1})|}{N^{1/2}} \leq \frac{\|d(\psi^{-1})\|}{N^{1/2}}  \leq 1,
\]
by \eqref{eq bound dpsi}.
On $(M \setminus V) \cap W$, the estimate \eqref{eq bound psi} implies that the second term in \eqref{eq est nabla X1H} is at most equal to
\[
 \frac{1}{\psi} \frac{\|\nabla_{v} \XoneH \|}{N} \leq  \frac{1}{\psi} \frac{F_1}{N} \leq 1.
\]
\end{proof}

\begin{lemma} \label{lem bound Apsi}
In the setting of the proof of Proposition \ref{prop choice family}, the operator $\widetilde{A}_1^{\psi}$ satisfies the pointwise estimate
\[
\| \widetilde{A}_1^{\psi}\| \leq \frac{3}{2}d_M N_{\psi}
\]
on $(M \setminus V) \cap W$.
\end{lemma}
\begin{proof}
By Lemma \ref{lem nabla X1H}, one has on $(M \setminus V) \cap W$, for all $k$,
\[
\left\| c(e_k)c\left(\nabla_{e_k} X_1^{\HH_{\psi}} \right) \right\| 
\leq \bigl\| \nabla_{e_k} X_1^{\HH_{\psi}}  \bigr\| 
\leq 2N_{\psi}.
\]
Similarly, 
one has on $(M \setminus V) \cap W$,
\[
\begin{split}
 \left|\tr\left(  \left. \nabla^{T^{1,0}M} X_1^{\HH_{\psi}} \right|_{T^{0,1}M}  \right) \right|
 	&= \left|\sum_{k=1}^{d_M}\left(\nabla^{T^{1,0}M}_{e_k^{1,0}} X_1^{\HH_{\psi}} , e_k^{1,0}\right)\right| \\
	&\leq \sum_{k=1}^{d_M} \bigl\| \nabla^{T^{1,0}M}_{e_k^{1,0}} X_1^{\HH_{\psi}} \bigr\|  \\
	&\leq 2d_M N_{\psi}.
\end{split}
\]
The estimate now follows from the definition \eqref{eq A tilde} of $\widetilde{A}_1^{\psi}$.
\end{proof}

\begin{lemma}  \label{lem mu Thj psi}
For all $j$, one has
\[
\langle \mu, T(h_j^{\psi})^* \rangle = \mu_j d(\psi^{1/2}) + \psi^{1/2} \langle \mu, Th_j^*\rangle. 
\]
\end{lemma}
\begin{proof}
Since $(h_j^{\psi})^* = \psi^{1/2}h_j^*$, one has for all $m \in M$,
\[
T_m (h_j^{\psi})^* = h_j^*(m) d_m(\psi^{1/2}) + \psi^{1/2} T_mh_j^*.
\]
Furthermore, one has $\langle \mu, h_j^*\rangle = \mu_j$.
\end{proof}

\begin{lemma} \label{lem est A2}
In the setting of the proof of Proposition \ref{prop choice family}, one has for all $j$,
\[
\| \langle \mu, T(h_j^{\psi})^* \rangle \| \cdot  \|V_j^{\psi}\| \leq 2N_{\psi},
\]
on $(M\setminus V) \cap W$.
\end{lemma}
\begin{proof}
By Lemma \ref{lem mu Thj psi}, we find that for all $j$,
\[
\frac{\| \langle \mu, T(h_j^{\psi})^* \rangle \| \cdot  \|V_j^{\psi} \| }{N_{\psi}}  \leq \frac{\|d(\psi^{1/2})\|}{\psi^{3/2}}  \frac{|\mu_j| \cdot  \|V_j\|}{N} + \frac{1}{\psi} \frac{\|\langle \mu, Th_j^*\rangle\| \cdot  \|V_j\|}{N}.
\]
By \eqref{eq bound psi}, the second of these terms is at most equal to $1$ on $(M\setminus V) \cap W$. The first term is equal to
\[
\frac{\|d(\psi^{-1})\|}{2}  \frac{|\mu_j| \cdot  \|V_j\|}{N}, 
\]
which is less than or equal to $1$ on $M\setminus V$, because of \eqref{eq bound dpsi}.
\end{proof}

\begin{lemma}  \label{lem Lie bracket psi}
For all vector fields $v \in \XX(M)$ and all $j$, one has
\[
 \bigl[v, \bigl((h^*_j)^{\psi}\bigr)^M - V_j^{\psi}\bigr]= \psi^{1/2}  \bigl[v, (h^*_j)^M - V_j\bigr] - v(\psi^{1/2}) V_j.
\]
\end{lemma}
\begin{proof}
The Leibniz rule for the Lie derivative of vector fields implies that for all $m, m' \in M$, 
 \begin{multline*}
 \bigl[v, \bigl((h^*_j)^{\psi}(m)\bigr)^M - V_j^{\psi}\bigr](m') =  \bigl[v, \psi(m)^{1/2}h^*_j(m)^M - \psi^{1/2}V_j\bigr](m') \\
 	= \psi(m)^{1/2} \bigl[v, h^*_j(m)^M\bigr](m') 
		- \left( \psi(m')^{1/2}[v,V_j](m') + v(\psi^{1/2}) (m') V_j(m') \right).
\end{multline*}
Taking $m = m'$ yields the desired equality.
\end{proof}

\begin{lemma} \label{lem est A3}
Let $v$ be one of the vector fields $e_k$ or $e_k^{1,0}$ in the proof of Proposition \ref{prop choice family}.
Then one has for all $j$, on $(M\setminus V) \cap W$,
\[
|\mu_j^{\psi}|   \bigl\| \bigl[v, \bigl((h^*_j)^{\psi}\bigr)^M - V_j^{\psi}\bigr]  \bigr\|  \leq 2 N_{\psi}.
\]
\end{lemma}
\begin{proof}
By Lemma \ref{lem Lie bracket psi}, one has for all $v$ vector fields $v \in \XX(M)$, and for all $j$, and $m \in (M\setminus V) \cap W$,
\begin{multline*}
\left( \frac{|\mu_j^{\psi}|   \bigl\| \bigl[v, \bigl((h^*_j)^{\psi}(m)\bigr)^M - V_j^{\psi}\bigr] \bigr\| }{N_{\psi}} \right)(m) \\
\leq 
	\left( \frac{1}{\psi}\frac{|\mu_j|    \bigl\| \bigl[v, h^*_j(m)^M - V_j\bigr]  \bigr\| }{N}\right)(m) + 
	\left( \frac{|v(\psi^{1/2})|}{\psi^{3/2}} \frac{|\mu_j| \|V_j\|}{N}\right)(m).
\end{multline*}
Now suppose $v$ is  one of the vector fields $e_k$ or $e_k^{1,0}$. Then, because of \eqref{eq bound psi}, the first term on the right hand side is at most equal to $1$. The second term is equal to
\[
\left( \frac{|v(\psi^{-1})|}{2} \frac{|\mu_j| \|V_j\|}{N}\right)(m).
\]
Because of \eqref{eq bound dpsi}, this expression is also at most equal to $1$.
\end{proof}


\bibliography{mybib}

\end{document}